\documentclass[11pt, a4paper]{amsart}
\usepackage[latin1]{inputenc}
\usepackage{amsmath, amsfonts, amssymb, amsthm, mathtools, mathrsfs,bm}
\usepackage{enumerate}
\usepackage[numbers]{natbib}
\usepackage{setspace}
\usepackage[usenames,dvipsnames]{xcolor}
\usepackage{xcolor}
\usepackage{tikz-cd}

\usepackage[pagebackref,colorlinks,linkcolor=blue,citecolor=blue,urlcolor=blue,hypertexnames=true]{hyperref}

\newtheorem{theorem}{Theorem}[section]
\newtheorem*{thmA}{Theorem A}

\newtheorem{lemma}[theorem]{Lemma}
\newtheorem{proposition}[theorem]{Proposition}
\newtheorem{corollary}[theorem]{Corollary}

\theoremstyle{definition}

\newtheorem{remark}[theorem]{Remark}

\newcommand{\ZZ}{\mathcal{Z}}
\newcommand{\OO}{\mathcal{O}}
\newcommand{\QQ}{\mathcal{Q}}
\newcommand{\N}{\mathbb{N}}
\newcommand{\Z}{\mathbb{Z}}
\newcommand{\Q}{\mathbb{Q}}
\newcommand{\R}{\mathbb{R}}
\newcommand{\C}{\mathbb{C}}

\newcommand{\mI}{\mathcal{I}}
\newcommand{\Sgr}{\mathcal{S}}

\newcommand{\K}{\mathcal{K}}	

\newcommand{\Aut}{\mathrm{Aut}}
\newcommand{\Autgr}{\mathrm{Aut}_{\text{gr}}}

\newcommand{\Cliff}[1]{\C \ell_{#1}}
\newcommand{\id}[1]{\mathrm{id}_{#1}}
\newcommand{\Ggr}{\operatorname{G}^{\text{gr}}}
\newcommand{\G}{\operatorname{G}}

\newcommand{\hE}{\hat{E}}
\newcommand{\bE}{\bar{E}}

\newcommand{\catU}{\mathcal{U}}

\newcommand{\bL}{\mathbb{L}}
\newcommand{\OSpec}{\mathscr{S}}
\newcommand{\wSpec}{\mathscr{S}^{\Omega}}
\newcommand{\SMod}{\mathscr{M}_S}
\newcommand{\ComS}{\mathrm{Com}_S}

\DeclareMathOperator{\hocolim}{hocolim}

\definecolor{darkpastelred}{rgb}{0.76, 0.23, 0.13}
\definecolor{darkred}{rgb}{0.55, 0.0, 0.0}
\definecolor{darkmagenta}{rgb}{0.55, 0.0, 0.55}
\definecolor{coolblack}{rgb}{0.0, 0.18, 0.39}
\definecolor{ceruleanblue}{rgb}{0.16, 0.32, 0.75}

\begin{document}
\title[bundles of strongly self-absorbing $C^*$-algebras]{{Computing  cohomology groups that classify bundles of strongly self-absorbing $C^*$-algebras}}
\author{Marius Dadarlat}
\author{James E. McClure}
\address{Department of Mathematics \\
Purdue University\\
West Lafayette, IN 47907, USA}
\author{Ulrich Pennig}
\address{School of Mathematics \\
Cardiff University \\
Cardiff \\
CF24 4AG\\
Wales\\
UK}

\thanks{J.E.McC. thanks the Lord for making his work possible}

\maketitle
\begin{abstract}
Locally trivial bundles of $C^*$-algebras with fibre $D \otimes \mathcal{K}$ for a strongly self-absorbing $C^*$-algebra $D$ over a finite CW-complex $X$ form a group $E^1_D(X)$ that is the first group of a cohomology theory $E^*_D(X)$. In this paper we compute these groups by expressing them in terms of ordinary cohomology and connective $K$-theory. To compare the $C^*$-algebraic version of $gl_1(KU)$ with its classical counterpart we also develop a uniqueness result for the unit spectrum of complex periodic topological $K$-theory.
\end{abstract}

\tableofcontents

\section{Introduction}
Continuous fields of $C^*$-algebras occur naturally as they correspond to bundles of $C^*$-algebras in the sense of topology \cite{BK:bundles}. Any $C^*$-algebra with Hausdorff primitive spectrum $X$ is a continuous field of simple $C^*$-algebras over $X$, \cite{Fell}, \cite{Dix:C*}. More importantly, continuous fields are used as versatile tools in several areas: $E$-theory \cite{Con-Hig:etheory}, deformations of the tangent groupoid of manifolds \cite{Con:noncomm}, \cite{Nest:tangent-groupoid}, strict deformation quantization \cite{Rieffel:quantization}, \cite{Landsman:book}, the Novikov conjecture and the Baum-Connes conjecture, \cite{Kas:inv}, \cite{HigKas:BC}, \cite{Tu:BC}, representation theory and index theory \cite{Higson:Mackey-analogy}, \cite{Gue-Hig:book}.

Particularly well-behaved examples of continuous fields {over a compact space $X$} can be obtained as the continuous sections  {$C(X,E)$} of a locally trivial bundle $E \to X$ with fibres isomorphic to a fixed $C^*$-algebra $A$. Such bundles are constructed by gluing together trivial bundles $U_i \times A$ over an open cover $(U_i)_{i \in I}$ of the base space $X$ using $1$-cocycles $\gamma_{ij} \colon U_i \cap U_j \to \Aut(A)$. Two such locally trivial bundles $E$ and $F$ over a compact Hausdorff space $X$ are isomorphic if and only if their section $C^*$-algebras {$C(X,E)$ and $C(X,F)$} are isomorphic via a $C(X)$-linear $*$-isomorphism. The proof of this property is an elementary exercise in bundle theory \cite{DaHiPhi}.

As the description by $1$-cocycles indicates, the classification of isomorphism classes of locally trivial bundles with fibre $A$ can be reduced to the classification of principal $\Aut(A)$-bundles over~$X$ up to isomorphism, and hence to the computation of the homotopy classes of maps from $X$ to the classifying space $B\Aut(A)$. Without further structure the homotopy type of the classifying space is typically very difficult to determine and the computation of homotopy sets such as $[X,B\Aut(A)]$ is out of reach. This changes quite drastically, however, if $A = D \otimes \K$, where $D$ belongs to the class of strongly self-absorbing $C^*$-algebras 
 (see Sec.~\ref{sect:two}) and $\K$ denotes the compact operators on an infinite-dimensional separable Hilbert space. In this case $\Aut(D \otimes \K)$ turns out to be an infinite loop space, which not only implies that $E^1_D(X) := [X,B\Aut(D \otimes \K)]$ is an abelian group,  the 1-group of the generalized cohomology theory $E^*_D(X)$, but also that it is amenable to methods from stable homotopy theory, \cite{DP1}, \cite{DP2}.

Remarkably, the group law on $E^1_D(X)$ coming from the infinite loop space structure of  $B\Aut(D \otimes \K)$ coincides with the operation induced by the tensor product of $D \otimes \K$-bundles. This led to a generalized Dixmier-Douady theory. Let us note that if $D=\mathbb{C}$, then  $E_{\mathbb{C}}^1(X)\cong H^3(X,\mathbb{Z})$ is  the home of the classic  Dixmier-Douady class. Strongly self-absorbing $C^*$-algebras \cite{paper:TomsWinter} are separable unital $C^*$-algebras $D$ defined by a crucial property that they share with the complex numbers $\C$.  Namely, there exists an isomorphism $D\to D\otimes D$ which is unitarily homotopic to the map $d\mapsto d\otimes 1_D$ \cite{Dadarlat-Winter:KK-of-ssa}, \cite{paper:WinterZStable}. Any strongly self-absorbing $C^*$-algebra $D$ is either stably finite or purely infinite. The latter condition is equivalent to $D\cong D\otimes \OO_\infty$, where  $\OO_\infty$ is the infinite Cuntz algebra. Due to recent progress in classification theory \cite{Winter:abel} we now have a complete list of all the strongly self-absorbing $C^*$-algebras that satisfy the Universal Coefficient Theorem (abbreviated~UCT) in KK-theory.

The main goal of this paper is to compute the group $E_D^1(X)$ for all strongly self-absorbing $C^*$-algebras $D$ in the UCT class. More precisely, we express $E_D^1(X)$ and its variants using connective $K$-theory and ordinary cohomology groups. While this question is interesting in itself and in view of direct connections with higher twisted K-theory \cite{Ulrich}, we draw additional motivation from the following recent development: a conjecture of Izumi from \cite{Izumi-conj}, \cite{Izumi-Matui} has been recently proved due to combined work of Meyer \cite{Meyer-KKG} and Gabe and Szab\'{o} \cite{Gabe-Szabo}. It asserts that for a countable torsion free amenable group $G$ and for a  Kirchberg algebra $D$, there is a bijection between the cocycle conjugacy classes  of outer actions of $G$ on $D\otimes K$ and the isomorphism classes of principal $\Aut(D \otimes K)$-bundles over the classifying space $BG$, i.e.\ the set of homotopy classes $[BG,B\mathrm{Aut}(D\otimes K)]$.

 It follows that if $BG$ admits a model as a finite CW-complex and if $D$ is a strongly self-absorbing Kirchberg algebra that satisfies the UCT, then the set of cocycle conjugacy classes of outer actions of $G$ on $D\otimes K$ forms a group with respect to the tensor product operation, and this group is isomorphic to $E^1_D(BG)$. Moreover, this group can be computed as explained below.

Let $k^*(X)$ denote the complex connective $K$-theory of the space $X$. For a finite based CW-complex $X$ with skeleta $X_i$, $\widetilde{k}^i(X)\cong \widetilde{K}^i(X,X_{i-2})$ and in particular $k^5(X)\cong K^1(X,X_3)$, see Proposition~\ref{k-th}.
For a set of primes $P\neq \emptyset$, {consider the $C^*$-algebra $M_P = \bigotimes_{p \in P} M_p(\C)^{\otimes \infty}$ and its $K$-theory ring
$\Z_P=\bigotimes_{p\in P}\Z[\frac{1}{p}]$ viewed as a subring of  $\Q$.}
Our main result is the following:
\begin{thmA} \label{thm:main-intro} Let $X$ be a finite  CW-complex and let $P$ be a nonempty set of prime numbers. There are  isomorphisms
\begin{itemize}
	\item[(a)] ${E}^1_{\ZZ}(X)\cong  H^3(X,\Z)\oplus k^5(X)$.
	\item[(b)] ${E}^1_{M_P}(X)\cong  H^1(X,(\Z_P)_+^\times){\oplus} H^3(X,\Z_P)\oplus k^5(X,\Z_P)$
	\item[(c)] ${E}^1_{\OO_\infty}(X)\cong \big(H^1(X, \Z/2)\times_{_{tw}}  H^3(X,\Z)\big)\oplus k^5(X)$.
    \item[(d)] ${E}^1_{M_P \otimes \OO_\infty}(X)\cong H^1(X,(\Z_P)_+^\times)\oplus \big(H^1(X, \Z/2)\times_{_{tw}} H^3(X,\Z_P)\big)\oplus k^5(X,\Z_P)$
\end{itemize}
The (twisted) multiplication on $ H^1(X;\Z/2) \times H^3(X,\Z_P)$ is given by
\begin{equation}\label{xcxc}
	(w, \tau) \cdot (w',\tau') = (w + w', \tau + \tau' + \beta_P(w \cup w'))
\end{equation}
for $w,w' \in H^1(X,\Z/2)$ and $\tau,\tau' \in H^3(X,\Z_P)$, where $\beta_P \colon H^2(X,\Z/2) \to H^3(X,\Z_P)$ is the composition of the Bockstein homomorphism $\beta$ with the coefficient map $H^3(X,\Z)\to H^3(X,\Z_P)$. The multiplicaton in (c) is just like in ~\eqref{xcxc} with   $\beta_P$
 replaced by $\beta \colon H^2(X,\Z/2) \to H^3(X,\Z)$.
\end{thmA}
\noindent The isomorphisms above are not natural.
The article is structured as follows: In Section 2.1 we recall the definition of strongly self-absorbing $C^*$-algebras, discuss their classification and give their $K$-theory groups. To each such algebra one can associate a commutative symmetric ring spectrum $KU^D$ (see \cite{DP2}). Its definition is recalled in Section~2.2. Just like a commutative ring has a group of invertible elements (or units), the spectrum $KU^D$ has an associated unit spectrum, whose definition is based on commutative $\mI$-monoids, see \cite{paper:Schlichtkrull}. All of these notions are reviewed in Section 2.3. As we will see in later chapters the groups classifying bundles of stabilised strongly self-absorbing $C^*$-algebras can be expressed in terms of connective $K$-theory $ku$. Therefore we recall the main properties of $ku$ needed in the rest of the paper in Section~2.4.

In Section~3 we discuss the classification of $C^*$-algebra bundles with fibre $D \otimes \K$ for a strongly self-absorbing $C^*$-algebra $D$. We review the construction of the cohomology theories $E^*_D(X)$ and its variants at the beginning of Section~3. As shown in Section~3.1 these theories all split off a low-degree summand that can be expressed via ordinary cohomology groups with a ``twisted multiplication''. The complement of this summand, called $h(X)$ in Section~3, is identified in Section~3.2 with $bsu^1_{\otimes}(X)$. {Due to the Adams-Priddy result on the uniqueness of $bsu$, \cite{AP}, we can express  $bsu^1_{\otimes}(X)$  in terms of connective $K$-theory (which we review in Sect.2.4).} The subtle point in this last part is that we have to deal with two constructions of the units of topological $K$-theory: $gl_1(KU^\C)$ and $gl_1(KU)$ for a commutative $S$-algebra $KU$ representing $K$-theory. We will show in Sections~4 and 5 that both of these give the same cohomology theory.

Since the proof of the uniqueness result in Section~5 requires a lot of heavy machinery from stable homotopy theory, Section~4 outlines the preliminaries that we need: We discuss the definition of commutative $S$-algebras and their unit spectra in Section~4.1, which also contains the definition of the commutative $S$-algebra $KU$ that will play a crucial role in Section~5.

Section~5 of this paper is devoted entirely to the proof of the uniqueness result which gives a natural isomorphism $gl_1(KU^\C)^*(X) \cong gl_1(KU)^*(X)$. In Section~5.1 we explain how to move from units of commutative symmetric ring spectra to units of commutative $S$-algebras. We use the obstruction theory of Goerss and Hopkins in Section~5.2 to show that uniqueness follows if we have a unital and multiplicative isomorphism between a commutative $S$-algebra model $G$ for $K$-theory and $KU$ in the homotopy category. In the remaining Sections~5.3 to 5.6 we construct an intermediate spectrum $\Sigma^\infty \C P^\infty_+[b^{-1}]$, which has the benefit that maps from it to $G$ and $KU$ can be constructed from maps on $\C P^n$ that are easily obtained from the multiplicative natural transformation.

\section{Symmetric ring spectra representing $K$-theory and its localisations}
Topological $K$-theory is a multiplicative cohomology theory and can be modelled as a commutative symmetric ring spectrum \cite{Joachim:highercoherences}. In this form it has a natural extension $KU^D$ that takes a strongly self-absorbing $C^*$-algebra $D$ as input. For any such $D$ the group $K_0(D)$ is actually a ring and if $D$ satisfies the UCT, then $K_0(D) \subseteq \Q$ as a subring. The spectrum $KU^D$ represents topological $K$-theory with coefficients in $K_0(D)$, i.e.\ a localisation of $K$-theory. We recall the definition of strongly self-absorbing $C^*$-algebras, the construction of $KU^D$ and its unit spectrum in the sense of stable homotopy theory in the next sections.

\subsection{Strongly self-absorbing $C^*$-algebras} \label{sect:two}

A $C^*$-algebra $A$ absorbs another $C^*$-algebra $D$ tensorially if $A \otimes D \cong A$. As was already observed in \cite{Kir:class, Phi:class}, tensorial absorption properties are crucial in the classification programme for separable simple nuclear $C^*$-algebras. This prompted an analysis of strongly self-absorbing $C^*$-algebras \cite{WinterToms:ssa}. By definition a unital $C^*$-algebra $D$ belongs to this class if there exists an isomorphism $D \to D \otimes D$ that is approximately unitarily equivalent to $d \mapsto d \otimes 1$.

The two Cuntz algebras $\OO_2$ and $\OO_\infty$ with two, respectively infinitely many generators (see \cite{Cuntz:On}) are strongly self-absorbing. Another prominent example in this class is the Jiang-Su algebra $\ZZ$ (see \cite{JiaSu:Z}), which can be viewed as the infinite-dimensional stably finite counterpart of~$\C$, while $\OO_{\infty}$ is its purely infinite version. We have an isomorphism $\OO_\infty \otimes \ZZ \cong \OO_\infty$. The unital $*$-homomorphisms $\C \to \ZZ \to \OO_\infty$ are $KK$-equivalences. $\OO_2$ is KK-contractible.

All self-absorbing $C^*$-algebras that satisfy the UCT, with the exception of $\C$ and $\OO_2$, can be obtained as tensor products of either $\ZZ$ or $\OO_\infty$ with an infinite UHF-algebra. The construction of these is a $C^*$-algebraic version of the localisation at a set of primes: For a prime $p$, let $M_{p}$ denote the infinite tensor product $M_{p} = M_p(\C)^{\otimes \infty}$. For a set of primes $P$ define $M_P$ to be
\[
	M_P=\bigotimes_{p\in P} M_{p}\ .
\]
If $P = \emptyset$, then we set $M_P=\C$. There is a dichotomy for strongly self-absorbing $C^*$-algebras: they are either stably finite or purely infinite. In the first case the algebra is isomorphic to either $\C$, $\ZZ$, or $M_P$ for some nonempty set $P$ of primes, in the second to $\OO_\infty$,  $M_P\otimes \OO_\infty$ for some nonempty set $P$ of primes or to $\OO_2$, \cite{Winter:abel}. If $D$ is purely infinite, then $D\otimes \OO_\infty \cong D$.

If $D$ is strongly self-absorbing and $D\neq \C$, then $D\cong D\otimes \ZZ$. In particular, we have $M_P\otimes \ZZ \cong M_P$. If $P\subset P'$ is an inclusion of subsets of prime numbers, then $M_{P'}\otimes M_{P}\cong M_{P'}$. In case $P$ is the set of all primes, then we denote $M_P$ by $\QQ$. The $C^*$-algebra $\QQ$ is called the universal UHF-algebra. The relationship between the various strongly self-absorbing $C^*$-algebras is illustrated in the following diagram:
\begin{equation} \label{eqn:ssa_uct}
\begin{tikzcd}[row sep=0.3cm]
					& \ZZ\ar[dd] \ar[r]	& M_P \ar[dd] \ar[r]				& \QQ \ar[dd] \ar[dr] \\
 \C \ar[ur] \ar[dr]	& 					& 								& 								& \OO_2 \\
					& \OO_\infty \ar[r]	& \OO_\infty \otimes M_P\ar[r]	& \OO_\infty \otimes \QQ \ar[ur]	
\end{tikzcd}
\end{equation}
An arrow $D \to D'$ in this diagram not only indicates a unital embedding, but also the property $D' \otimes D \cong D'$.

For strongly self-absorbing $C^*$-algebras $D$ in the UCT class the $K$-theory groups {can be computed as follows, \cite{WinterToms:ssa}: $K_1(D)=0$} and as mentioned above, the group $K_0(D) \cong K_0(\OO_\infty \otimes D)$ has a natural unital commutative ring structure with multiplication induced by the isomorphism $D \otimes D \cong D$. Let $\Z_p=\Z[\frac{1}{p}]$ denote the localization of $\Z$ away from $p$ and $\Z_{(p)}$ the localization of $\Z$ at $p$. Then we have natural ring isomorphisms:
\begin{gather*}
	K_0(\C) \cong K_0(\ZZ) \cong K(\OO_\infty) \cong \Z \ ,\\
	K_0(M_P) \cong  K_0(M_P\otimes \OO_\infty) \cong \bigotimes_{p\in P}\Z_p=:\Z_P \ ,\\
	K_0(\OO_2)=0\ .
\end{gather*}
Mirroring the localisation of integers, we denote the algebra $M_P$ by $M_{(p)}$ in case $P$ is the set of all primes different from $p$. Thus: $K_0(M_{p})\cong \Z_p$, $K_0(M_{(p)})\cong \Z_{(p)}$, and $K_0(\QQ)\cong \Q$.

We will denote the invertible elements of the commutative ring $K_0(D)$ either by $K_0(D)^{\times}$ or by $GL_1(K_0(D))$. Note that $K_0(D)$ is an ordered group with positive cone given by the elements represented by classes of projections in $D \otimes \K$ (as opposed to formal differences). The subgroup of positive elements of $K_0(D)^{\times}$ is then denoted by $K_0(D)^{\times}_+$. Since the order structure is trivial if $D$ is purely infinite, we have $K_0(D)^{\times}_{+}=K_0(D)^{\times}$ in this case.

\subsection{The symmetric spectra $KU^D$}
A $\Z/2\Z$-grading on a $C^*$-algebra $A$ is an automorphism $\gamma \in \Aut(A)$ with $\gamma^2 = \id{A}$. We call the pair $(A, \gamma)$ a graded $C^*$-algebra and will often suppress the grading in the notation. Any $\Z/2\Z$-graded $C^*$-algebra $A$ has a Banach space decomposition $A \cong A^0 \oplus A^1$ with
\[
	A^0 = \{ a \in A \ | \ \gamma(a) = a \} \qquad \text{and} \qquad A^1 = \{ a \in A \ | \ \gamma(a) = -a \}
\]
such that the even part $A^0$ is a $C^*$-subalgebra and $A^i \cdot A^j \subset A^{i + j}$, where the supscript is taken modulo $2$ here. The elements $a \in A^i$ are said to be homogeneous and have degree $i$, which we denote by $\deg(a) = i$. If $\gamma = \id{A}$, then we call $A$ trivially graded. The (minimal) graded tensor product of two graded $C^*$-algebras $A$ and $B$ is a completion of the algebraic tensor product $A \odot B$ with the multiplication and involution defined on homogeneous elements by
\[
		(a \otimes b) \cdot (a' \otimes b') = (-1)^{\deg(a')\cdot \deg(b)} aa' \otimes bb' \qquad \text{and} \qquad (a \otimes b)^* = (-1)^{\deg(a)\cdot \deg(b)} a^* \otimes b^*
\]
The tensor flip $A \otimes B \to B \otimes A$ has to be decorated with the sign $(-1)^{\deg(a)\cdot \deg(b)}$ as well to be a $*$-isomorphism.

Two graded $C^*$-algebras will play a central role in the following construction: The Clifford algebra $\Cliff{n}$ is the unital $C^*$-algebra with self-adjoint generators $\{e_1, \dots, e_n\}$ and relations
\[
	e_i e_j + e_j e_i = \delta_{ij}\,1\ ,
\]
where $\delta_{ii}$ is the Kronecker delta. The grading $\gamma$ is defined on generators by $\gamma(e_i) = -e_i$, i.e.\ the elements $e_i$ are odd. These structures turn $\Cliff{n}$ into a finite-dimensional graded $C^*$-algebra.

The other non-trivially graded $C^*$-algebra that we will encounter is the graded suspension algebra $\Sgr = C_0(\R)$ equipped with the grading by odd and even functions. This algebra can be equipped with a coassociative and cocommutative comultiplication $\Delta \colon \Sgr \to \Sgr \otimes \Sgr$ that has a counit $\epsilon \colon \Sgr \to \C$ \cite[p.~94]{Joachim:highercoherences}.

Let $D$ be a strongly self-absorbing $C^*$-algebra, considered as a trivially graded algebra. It was shown in \cite{DP2} that the sequence of spaces $(KU^D_n)_{n \in \N_0}$ given by
\[
	KU^D_n = \hom_{\text{gr}}(\Sgr, (\Cliff{1} \otimes D \otimes \K)^{\otimes n})
\]
and equipped with the point-norm topology forms a commutative symmetric ring spectrum representing the cohomology theory $X \mapsto K_*(C(X) \otimes D)$. The multiplicative structure on $KU^D$ is induced by
\[
	KU^D_n \wedge KU^D_m \to KU^D_{n+m} \qquad, \qquad \varphi \wedge \psi \mapsto (\varphi \otimes \psi) \circ \Delta\ .
\]
Bott periodicity gives an element in $\hom_{\text{gr}}(\Sgr, C_0(\R) \otimes \Cliff{1})$ and hence by extension with a rank $1$-projection $1 \otimes e \in D \otimes \K$ an element in
\[
	\text{Map}_*(S^1, KU^D_1) \cong  \hom_{\text{gr}}(\Sgr, C_0(\R) \otimes \Cliff{1} \otimes D \otimes \K)
\]
(see \cite[Sec.~4.1]{DP2} for details). The unit map $\eta_n \colon S^n \to KU^D_n$ of the ring spectrum is constructed as an $n$-fold power of $\eta_1$ (using the above multiplication). As usual combining the unit maps with the multiplication gives the structure maps
\begin{equation} \label{eqn:structure_map}
	S^1 \wedge KU^D_n \to KU^D_{n+1}
\end{equation}
{Note that the space $KU^D_0 = \hom_{\text{gr}}(\Sgr, \C)\simeq S^0$ contains only two points. One is given by the zero homomorphism, the other one is given by the evaluation at $0$, which is the only evaluation that respects the grading. }
 Nevertheless, as we will see in Lem.~\ref{lem:multiplicative} all adjoints of the structure maps in degree $\geq 1$ are weak homotopy equivalences. Symmetric spectra with this property are called \emph{positive $\Omega$-spectra}. In fact, more is true:
\begin{lemma} \label{lem:multiplicative}
	Let $D$ be a strongly self-absorbing $C^*$-algebra. The spectrum $KU^D$ is a positive $\Omega$-spectrum that represents $X \mapsto K_*(C(X) \otimes D)$ as a multiplicative cohomology theory.
\end{lemma}

\begin{proof}
	For $n \geq 1$ the adjoint $KU^D_n \to \Omega KU^D_{n+1}$ of \eqref{eqn:structure_map} is a weak equivalence by \cite[Thm.~4.2]{DP2}. Thus, $KU^D$ is a positive $\Omega$-spectrum. By \cite[Thm.~4.7]{paper:Trout} there are natural isomorphisms
	\[
		[X, KU^D_n] \cong [\Sgr, C(X) \otimes \Cliff{n} \otimes D \otimes \K]_{\text{gr}} \cong KK(\C, C(X) \otimes \Cliff{n} \otimes D) \cong K_n(C(X) \otimes D)\ ,
	\]
	for $n\geq 1$, where $[\,\cdot\, , \,\cdot\,]_{\text{gr}}$ denotes the homotopy classes of graded homomorphisms. It remains to be checked that this is compatible with the multiplicative structure. Using Bott periodicity and a suspension argument this can be reduced to the question whether $KU^D_2 \wedge KU^D_2 \to KU^D_4$ implements the multiplication in $K_0(C(X) \otimes D) \cong K_2(C(X) \otimes D)$.
	
	Note that $\Cliff{2} \cong M_2(\C)$ if the right hand side is equipped with the diagonal/off-diagonal grading. Let $p_1, p_2, q_1, q_2 \in C(X) \otimes D \otimes \K$ be projections and consider
	\[
		\varphi_{p_i,q_i} \colon \Sgr \to \Cliff{2} \otimes C(X) \otimes D \otimes \K \quad , \qquad f \mapsto \epsilon(f)\,
		\begin{pmatrix}
			p_i & 0 \\
			0 & q_i
		\end{pmatrix}\ .
	\]
	By \cite[p.~303]{paper:Trout} this represents the element $[p_i] - [q_i] \in K_0(C(X) \otimes D)$ in $[X,KU_2^D]$. We have $\Cliff{2} \otimes \Cliff{2} \cong \Cliff{2} \otimes M_2(\C)$ with $M_2(\C)$ trivially graded. Moreover, $(\epsilon \otimes \epsilon) \circ \Delta = \epsilon$, since $\epsilon$ is a counit. Hence, we obtain
	\[
		((\varphi_{p_1,q_1} \otimes \varphi_{p_2,q_2}) \circ \Delta)(f) = \epsilon(f)\,
		\begin{pmatrix}
			p_1 \otimes p_2 \oplus q_1 \otimes q_2  & 0 \\
			0 & p_1 \otimes q_2 \oplus q_1 \otimes p_2
		\end{pmatrix}\ ,
	\]
	which indeed represents the class $([p_1] - [q_1]) \cdot ([p_2] - [q_2]) \in K_0(C(X) \otimes D)$.
\end{proof}

\subsection{The unit spectrum of $KU^D$}\label{subsec: units-KUd}
{Following Schlichtkrull \cite{paper:Schlichtkrull}}, we give a brief outline of how to define the units of a commutative symmetric ring spectrum. This is based on the following model for $E_{\infty}$-spaces: Let $\mI$ be the category with objects the sets $\mathbf{n} = \{1,\dots,n\}$ for $n \in \N_0$ (with $\mathbf{0} = \emptyset$) and injective maps as morphisms. This is a symmetric monoidal category, where the tensor product is given by $\mathbf{n} \sqcup \mathbf{m} = \{1, \dots, n+m\}$ on objects and on morphisms by identifying $\mathbf{n}$ with $\{1,\dots, n\} \subset \mathbf{n} \sqcup \mathbf{m}$ and $\mathbf{m}$ with $\{n+1, \dots, n+m\} \subset \mathbf{n} \sqcup \mathbf{m}$. The object $\mathbf{0}$ is the monoidal unit and the symmetry $\mathbf{n+m} \to \mathbf{m+n}$ is given by a block permutation.

An $\mI$-space is a functor $X \colon \mI \to \mathcal{T}$ from $\mI$ to the category of based topological spaces. A morphism of $\mI$-spaces is a natural transformation. The category of $\mI$-spaces also has a symmetric monoidal structure: For two given $\mI$-spaces $X$ and $Y$, the tensor product $X \boxtimes Y$ is defined as the left Kan extension of the $\mI^2$-space $X \times Y$ along $\sqcup \colon \mI \times \mI \to \mI$. This definition implies that a morphism $X \boxtimes Y \to Z$ of $\mI$-spaces $X, Y, Z$ is the same as a natural transformation
\[
	X(\mathbf{n}) \times Y(\mathbf{m}) \to Z(\mathbf{n} \sqcup \mathbf{m})
\]
of $\mI^2$-spaces. A (commutative) $\mI$-monoid is a (commutative) monoid in the category of $\mI$-spaces. Note that the category $\mI$ is denoted by $\mathbb{I}$ in \cite{paper:LindUnits} and $\mI$-monoids are called $\mathbb{I}$-FCPs. If $X$ is a commutative $\mI$-monoid, then
\[
	X_{h\mI} = \hocolim_\mI X
\]
is an $E_\infty$-space \cite[Rem.~4.3]{paper:LindUnits}. One should think of the homotopy colimit of an $\mI$-space as the homotopy type modelled by it.

Let $(R_n)_{n \in \N_0}$ be a commutative symmetric ring spectrum and define $(\Omega^\bullet R)(\mathbf{n}) = \Omega^n R_n$. This extends to an $\mI$-space $\Omega^\bullet R$, which should model the infinite loop space underlying the spectrum $R$. However, there is a caveat here, since the homotopy groups of $(\Omega^\bullet R)_{h\mI}$ and~$R$ do not necessarily agree (see also \cite[Rem.~2.1]{paper:LindUnits}). If $R$ is represented by a positive $\Omega$-spectrum, which we will assume for the rest of this paragraph, then this problem does not arise. In analogy to \eqref{eqn:space_of_units} we define the $\mI$-space of units $(\Omega^\bullet R)^\times(\mathbf{n})$ of $R$ to be the subspace of $\Omega^n R_n$ consisting of those based maps $f \colon S^n \to R_n$ which are stably invertible in the sense that there exists $g \colon S^m \to R_m$ such that
\[
	\mu \circ (f \wedge g) \colon S^{n+m} \to R_{n+m}
\]
is homotopic to the unit map of the ring spectrum, where $\mu \colon R_n \wedge R_m \to R_{n+m}$ denotes its multiplication. The smash product $(f,g) \mapsto \mu \circ (f \wedge g)$ gives $(\Omega^\bullet R)^\times$ the structure of a commutative $\mI$-monoid and
\[
	(\Omega^\bullet R)^\times_{h\mI} \simeq GL_1(R)\ .
\]
Thus, the space underlying $(\Omega^\bullet R)^\times$ is $GL_1(R)$, but the commutative $\mI$-monoid reveals an $E_\infty$-structure on it. To obtain the spectrum $gl_1(R)$ from this, we can turn $(\Omega^\bullet R)^\times$ into a $\Gamma$-space and apply an infinite loop space machine to it turning it into a weak $\Omega$-spectrum.
For details we refer the reader to \cite{paper:Schlichtkrull} or to \cite[Construction~12.1 and Def.~12.5]{paper:LindUnits}.


{Applying this functor to $KU^\C$ gives a first way of defining the unit spectrum of topological $K$-theory and which we denote by $gl_1(KU^\C)$.
We will define the unit spectrum of topological $K$-theory in a second second way  in Section~\ref{Sect.4} and denote the corresponding object by $gl_1(KU)$. The two spectra $gl_1(KU^\C)$ and
 $gl_1(KU)$ will be compared in Section~\ref{Sect.5}.}

 We finish this part by listing some comparison results among the spectra $gl_1(KU^D)$. Recall that a {\it level equivalence} is a map of symmetric spectra which induces a weak equivalence of the $n$-th spaces of the spectra for all $n\geq 0$ (\cite[Definition 6.1(i)]{MMSS}). A level equivalence is a stable equivalence (\cite[bottom of page 466]{MMSS}). The converse is not true in general but it is true if the spectra are $\Omega$-spectra (\cite[Lemma 8.11]{MMSS}).

The spectra $KU^D$ are positive $\Omega$-spectra by Lem.~\ref{lem:multiplicative} and therefore fibrant objects in the model category of commutative symmetric ring spectra by \cite[Thm.~14.2]{MMSS}. By \cite[Lem.~13.5]{paper:LindUnits} the functor $gl_1$ preserves weak equivalences between fibrant objects. The $*$-homomorphisms $\C \to \ZZ$ and $\C \to \OO_{\infty}$ induce level and therefore stable equivalences of commutative symmetric ring spectra $KU^\C \to KU^\ZZ$ and $KU^\C \to KU^{\OO_\infty}$. Given any set of primes $P$ the unital $*$-homomorphism $M_P \to M_P \otimes \OO_\infty$ gives a stable equivalence $KU^{M_P} \to KU^{M_P \otimes \OO_\infty}$. From these we therefore obtain equivalences of weak $\Omega$-spectra
\begin{align}\label{eq:useful}
	gl_1(KU^\C) & \simeq gl_1(KU^\ZZ) \simeq gl_1(KU^{\OO_\infty})\ , \\
	gl_1(KU^{M_P}) & \simeq gl_1(KU^{M_P \otimes \OO_{\infty}}) \ .
\end{align}

\subsection{Connective K-theory}
Let $ku$ be a spectrum representing connective $K$-theory. It is the connective cover i.e.\ the $(-1)$-connected cover of the spectrum $KU$ of complex periodic K-theory. There is a map $ku \to KU$ of spectra which induces isomorphisms on homotopy groups in non-negative degrees while $\pi_i(ku)=0$ for $i<0$. Thus, the homotopy types of the first few spaces forming the spectrum $ku$ are:
\[
	BU\times \Z, \, U, \, BU, \, SU, \, BSU, \,  BBSU\, \dots
\]
The spectrum $bsu_{\oplus}$ is the $3$-connected cover of $ku$ (or of $KU$) (see \cite{AP}). Hence, we can identify $bsu_{\oplus}$ with $\Sigma^4 ku$. Let $k^*$ be the cohomology theory associated to $ku$. By our observations there is a natural isomorphism	
\begin{equation}\label{bbssuu}
bsu_{\oplus}^n(X) \cong k^{n+4}(X),\,\, \text{for all}\, n\geq 0.
\end{equation}
For our purposes, it is useful to address the question of computing $bsu_{\oplus}^1(X)\cong  k^5(X)$. Let $H\Z$ denote the usual Eilenberg-Mac Lane spectrum representing ordinary cohomology. The Bott operation $b$ induces a  cofiber sequence of spectra
\[
\begin{tikzcd}
	{\Sigma^2 ku} \ar[r, "b"] & ku \ar[r]  &H\Z
\end{tikzcd}
\]
and hence  a long exact sequence
\[
\begin{tikzcd}
 0 \ar[r] & k^3(X) \ar[r,"b"] & k^1(X) \arrow[r,shift right=2pt] & H^1(X,\Z) \ar[l,shift right=2pt] \ar[dll] \\
 0 \ar[r] & k^4(X) \ar[r,"b" below] & k^2(X) \ar[r,"c_1"] & H^2(X,\Z)\ar[dll] \\
 0 \ar[r] & k^5(X) \ar[r,"b" below] & k^3(X) \ar[r] & H^3(X,\Z)\ar[dll] \\
 & k^6(X) \ar[r,"b" below] & k^4(X) \ar[r] & H^4(X,\Z)	
\end{tikzcd}
\]
We see that $k^1(X)\cong K^1(X)$ and $k^3(X)\cong \mathrm{ker} \big(K^1(X)\to H^1(X,\Z))$ so that
\begin{equation}\label{eqk5}
	k^5(X) \cong \mathrm{ker} \big(k^3(X)\to H^3(X,\Z)\big)\ .
\end{equation}
The maps $k^i(X)\to H^i(X,\Z)$ are induced by the delopings of determinant map $\det \colon U \to U(1)$. The map $c_1$ identifies with the first Chern class and is surjective.

For a different description of $k^5(X)$, one verifies that if $X$ is a finite  CW-complex with i-skeleta $X_i$, $i\geq 0$, then $k^3(X)\cong K^1(X,X_1)$, $k^5(X)\cong K^1(X,X_3)$ is isomorphic to the kernel of the restriction map $K^1(X,X_1)\to K^1(X_3,X_1)$.
Indeed, for the benefit of the reader we include the following proposition known to the experts:
\begin{proposition}\label{k-th}
  If $X$ is a finite  CW-complex with base point, the pair $X_{i-2}\subset X$ induces an isomorphism of reduced theories $\widetilde{k}^i(X)\cong \widetilde{k}^i(X, X_{i-2})\cong \widetilde{K}^i(X, X_{i-2})$.
\end{proposition}
\begin{proof}
  The long exact sequence
  \[
	\begin{tikzcd}
		{} \ar[r] & \widetilde{k}^{i-1}(X_{i-2})\ar[r] & \widetilde{k}^i(X, X_{i-2})\ar[r] & \widetilde{k}^i(X)\ar[r] \ar[r] & \widetilde{k}^{i}(X_{i-2})\ar[r] &{}
	\end{tikzcd}
	\]
gives the isomorphism $\widetilde{k}^i(X)\cong \widetilde{k}^i(X, X_{i-2})$ since $\widetilde{k}^{i-1}(X_{i-2})=\widetilde{k}^{i}(X_{i-2})=0$ as the space $bu(r)$ is $(r-1)$-connected by construction and hence
$\widetilde{k}^{r}(Y)=[Y,bu(r)]=0$ if $\mathrm{dim}(Y)\leq r-1$.
Next, the exact sequence
\[
	\begin{tikzcd}
	\widetilde{H}^{j-1}(Y,\Z)\ar[r] & \widetilde{k}^{j+2}(Y)\ar[r] & \widetilde{k}^{j}(Y)\ar[r] \ar[r] & \widetilde{H}^{j}(Y,\Z)
	\end{tikzcd}
	\]
shows that $\widetilde{k}^{j+2}(Y)\cong\widetilde{k}^{j}(Y)$ if $Y$ is $j$-connected and hence
\[\widetilde{k}^{j+2}(Y)\cong\widetilde{k}^{j}(Y)\cong \widetilde{k}^{j-2}(Y)\cong \cdots \cong \widetilde{K}^{j-2m}(Y)\cong \widetilde{K}^{j}(Y)\] for $2m\geq j$. Aplying this isomorphism for $Y=X/X_{i-2}$ (which is $(i-2)$-connected) one obtains that
$\widetilde{k}^i(X, X_{i-2})\cong \widetilde{K}^i(X, X_{i-2})$.
\end{proof}
We refer the reader to \cite{Segal:K-homology}, \cite{DN:shape} and \cite{DadMcClure:When} for other applications of connective K-theory in operator algebras.

\section{Units of $K$-theory and the classification of bundles}
Let $D$ be a stably finite strongly self-absorbing $C^*$-algebra satisfying the UCT. From \eqref{eqn:ssa_uct} we know that $D \cong M_P$ for some set $P$ (possibly empty) of prime numbers. It was shown in \cite{DP1, DP2} that the isomorphism classes of locally trivial $C^*$-algebra bundles with fibre $D \otimes \K$ form a group with respect to the fibrewise tensor product. Each of these groups is part of a cohomology theory closely related to the one represented by the spectrum $gl_1(KU^D)$. The reason for the existence of these theories is an infinite loop space structure on the spaces
\[
	\Aut_0(D\otimes \K), \,\,\Aut(D\otimes \K),\,\, \Aut(D\otimes \OO_\infty \otimes  \K) \quad\text{and}\quad  \Aut_{\text{gr}}(\Cliff{1} \otimes D\otimes  \K) \ ,
\]
where $\Aut_0(\,\cdot\,)$ denotes the identity component of the automorphism group and $\Aut_{\text{gr}}(\,\cdot\,)$ are the automorphisms that preserve the grading. In fact, each of these spaces is an $E_\infty$-space that can be modelled by a commutative $\mI$-monoid naturally associated to $D$. These $\mI$-monoids and our notation for the associated cohomology theory are listed in Table~\ref{tab:list_of_coh}.
\begin{table}[h]
{\renewcommand{\arraystretch}{1.2}
\begin{tabular}{|rcl|l|}
	\hline
	\multicolumn{3}{|c|}{commutative $\mI$-monoid} & cohomology theory \\
	\hline
	$\bar{\G}_D(\mathbf{n})$ & $=$ & $\Aut_0(( D \otimes \K)^{\otimes n})$ & $\bar{E}^*_{D}(X)$ \\
	$\G_D(\mathbf{n})$ & $=$ & $\Aut(( D \otimes \K)^{\otimes n})$ & $E^*_{D}(X)$ \\
	$\G_{D\otimes \OO_\infty}(\mathbf{n})$ & $=$ & $\Aut(( D\otimes \OO_\infty \otimes \K)^{\otimes n})$ & $E^*_{D\otimes \OO_\infty}(X)$ \\
	$\Ggr_D(\mathbf{n})$ & $=$ & $\Autgr((\Cliff{1} \otimes D \otimes \K)^{\otimes n})$ & $\hE^*_{D}(X)$ \\
	\hline
\end{tabular}
\vspace*{3mm}
\caption{\label{tab:list_of_coh}Cohomology theories associated to the automorphism groups}
}
\end{table}

We refer the reader to \cite[Sec.~4.2]{DP2} and to \cite{DP4} for further details about their construction. Note that each of these $\mI$-monoids takes values in topological groups, and if $G$ denotes one of them, then $\mathbf{n} \to BG(\mathbf{n})$ (defined by taking classifying spaces levelwise) is a commutative $\mI$-monoid as well. It was shown in \cite[Thm.~3.6]{DP2} that the spectrum obtained from $BG$ is in fact the shifted version of the spectrum associated to $G$ itself. This is not obvious, since the classifying space delooping could a priori be different from the one produced by the infinite loop space machine. As a consequence we have $[X,B\Aut(D \otimes \K)] \cong E^1_D(X)$ and similarly for the other variations {listed in the table above.}

We will need the following result contained in Theorem 5.4 from \cite{DP4}. Recall that
the unit spectrum of $KU^D$, denoted by $gl_1(KU^{D})$ was discussed
 in subsec. \ref{subsec: units-KUd}.
\begin{theorem}
 Let $D$ be a stably finite strongly self-absorbing $C^*$-algebra that satisfies the UCT.
 There is a natural action of $\Ggr_D$  on the ring spectrum $KU^D$  which induces a map
 \begin{equation} \label{eqn:map_of_Gamma}
		\Gamma(\Ggr_D) \to \Gamma((\Omega^\bullet KU^D)^\times)
	\end{equation}
	of the associated $\Gamma$-spaces which  {in its turn} induces an equivalence  of the underlying connective spectra in the stable homotopy category. In particular we obtain a natural isomorphism
\begin{equation} \label{eqn:key-unit}
		\hE^*_{D}(X) \cong gl_1(KU^{D})^*(X)
	\end{equation}
for $X$ a finite CW-complex.
\end{theorem}

\subsection{Splitting results}\label{sec:three}
The main result in this section is Prop.~\ref{prop:basic}, in which we show that each of the cohomology theories defined above splits off a low-degree ordinary cohomology group. It suffices to treat the case where $D$ is stably finite, because a key result of \cite{DP4} establishes the isomorphism
\[
	E^1_{D\otimes \OO_\infty}(X)\cong \hE^1_{D}(X)\ .
\]
We have revisited  in \cite{DP4} the following result from \cite{paper:DonovanKaroubi}, \cite{paper:Parker-Brauer}.
\begin{proposition} \label{prop:Z2}
For any finite CW-complex $X$, $\hE_\C^1(X)\cong H^1(X,\Z/2) \times_{_{tw}}  H^3(X,\Z)$ with group structure:
\[
	(w, \tau) \cdot (w',\tau') = (w + w', \tau + \tau' + \beta(w \cup w'))
\]
for $w,w' \in H^1(X,\Z/2)$ and $\tau,\tau' \in H^3(X,\Z)$, where $\beta \colon H^2(X,\Z/2) \to H^3(X,\Z)$ is the Bockstein homomorphism.
\end{proposition}

Let us recall the following two theorems  from \cite{DP4}:
\begin{theorem}\label{thm:Z2}
    {Let $X$ be a finite CW-complex and let $D$ be a stable finite strongly self-absorbing $C^*$-algebra satisfying the UCT. The groups $E^1_{D}(X)$ and $\hE^1_{D}(X)$ fit into a short exact sequence}
    	\[
	\begin{tikzcd}\label{short}
		0 \ar[r] & E^1_{D}(X) \ar[r] & \hE^1_{D}(X) \ar[r ,"\delta_0"] & H^1(X, \Z/2) \ar[r] & 0.
	\end{tikzcd}
	\]
	If $L$ is a real line vector bundle on $X$ with associated Clifford bundle $\Cliff{L}$, then $\delta_0(\Cliff{L}\otimes D \otimes \mathcal{K})=w_1(L)$, {where $w_1(L)$ is the first Stiefel-Whitney class of $L$.}
\end{theorem}

\begin{theorem}\label{thm:basic}
{Let $X$ be a finite CW-complex and let $D$ be a stably finite strongly self-absorbing $C^*$-algebra satisfying the UCT.} Then there is an isomorphism of groups
\[\hE^1_{D}(X)\cong   H^1(X;\Z/2) \times_{_{tw}} E^1_{D}(X)\]
with multiplication on the direct product  $H^1(X;\Z/2) \times E^1_{D}(X)$ given by
\[
	(w, \tau) \cdot (w',\tau') = (w + w', \tau + \tau' + j_P\circ\beta(w \cup w'))
\]
for $w,w' \in H^1(X,\Z/2)$ and $\tau,\tau' \in E^1_{D}(X)$, where $j_P \colon   E^1_{\C}(X)\to E^1_{D}(X)$ is
the map induced by the unital $*$-homomorphism $\C\to D$ and we identify
$E^1_{\C}(X)\cong H^3(X,\Z)$.

\end{theorem}
\begin{remark}\label{remark:review} Just like in \cite{DP4}, we use  here the following basic fact, \cite[p.93]{Brown:book-cohomology}.
Suppose that
\[
	\begin{tikzcd}
		0 \ar[r] & A \ar[r,"i"] & E \ar[r,"\pi"] & G \ar[r] & 0
	\end{tikzcd}
	\]
	is an extension of abelian groups and $\sigma:G \to E$ is a map such that $\sigma(0)=0$ and $\pi\circ \sigma=\mathrm{id}_E$. Let $c:G\times G \to A$ be the normalized $2$-cocycle defined by
	$i(c(g,h))=\sigma(gh)\sigma(h)^{-1}\sigma^{-1}(g)$, $g,h\in G$. Then the group $E$ is isomorphic to $G\times A$ endowed with the group law:
	\[(g,a)(g',a')=(g+g',a+a'+c(g,g')).\]
\end{remark}

{The coefficients of the cohomology theory $E_D^*(X)$ are given by the homotopy groups of
$\Aut(D\otimes \K)$ computed in \cite[Thm.2.18]{DP1}:
\begin{equation}\label{homotopy}
E_D^{-i}(*)\cong \pi_i(\Aut(D\otimes \K))=
\begin{cases}
K_0(D)^{\times}_{+},\,\,\text{if}\,\, i=0\\
K_i(D),\,\,\text{if}\,\, i\geq 1.\\
\end{cases}
\end{equation}
If $D$ satisfies the UCT, then $K_{i}(D)=0$ if $i$ is odd and
in particular if $P\neq \emptyset$ is a set of prime numbers, then
\begin{equation}\label{homotopy1}
E_{M_P}^{-i}(*)=
\begin{cases}
(\Z_{P})^{\times}_{+},\,\,\text{if}\,\, i=0\\
\Z_P,\,\,\text{if}\,\, i=2k>0,\\
0, \,\,\text{if}\,\, i=2k+1.
\end{cases}
\end{equation}
}
Let $P$ be a set of prime numbers and let $\mathcal{A} \to X$ be a locally trivial $C^*$-algebra bundle with fibre $M_P\otimes \K$. Let $(U_i)_{i \in I}$ be a trivialising open cover of $X$ for $\mathcal{A}$ and denote the transition maps by $\varphi_{ij} \colon U_i \cap U_j \to \Aut(M_P \otimes \K)$. If we apply the $K_0$-functor to $\varphi_{ij}$ we obtain a $1$-cocycle
\[
	K_0(\varphi_{ij}) \colon U_i \cap U_j \to \Aut(K_0(M_P \otimes \K)) \cong GL_1(\Z_P)\ ,
\]
which factors through the group of positive invertible elements $(\Z_P)^\times_+ \subset GL_1(\Z_P)$. The cocycle only depends on the isomorphism class of $\mathcal{A}$ and is compatible with tensor products, which gives us a group homomorphism (see also \cite[Prop.~2.3]{DP3})
\[
	\delta_0 \colon E^1_{M_P}(X) \to H^1(X, (\Z_P)_+^\times)\ .
\]

\begin{proposition}\label{thm:MP}
	Let $X$ be a finite CW-complex and let $P\neq \emptyset$ be a set of prime numbers. The group $E^1_{M_P}(X)$ fits into a short exact sequence
	\[
	\begin{tikzcd}
		0 \ar[r] & \bar{E}^1_{M_P}(X) \ar[r] & E^1_{M_P}(X) \ar[r,"\delta_0"] & H^1(X, (\Z_P)^\times_+) \ar[r] & 0
	\end{tikzcd}
	\]
	and this sequence splits (not naturally). In particular, there is a non-natural isomorphism
	\[
		E^1_{M_P}(X) \cong H^1(X, (\Z_P)^\times_+) \oplus \bar{E}^1_{M_P}(X)\ .
	\]
\end{proposition}

\begin{proof}  By ~\eqref{homotopy1}, $\pi_0(\Aut(M_P\otimes \K))\cong K_0(M_P)^\times_{+}\cong (\Z_P)_+^\times$.
By \cite[Cor.2.19]{DP2}, there is an exact sequence
\[1\to\Aut_0(M_P\otimes \K) \to \Aut(M_P\otimes \K) \to K_0(M_P)^\times_{+}\to1.\]
We will first show that this sequence splits.
	  We will do so by constructing a homomorphism $\gamma:(\Z_P)_+^\times \to \Aut(M_P \otimes \K)$ that lifts $\Aut(M_P \otimes \K) \to \pi_0(\Aut(M_P \otimes \K))$. Note that
	\begin{equation} \label{eqn:primefactor}
		(\Z_P)_+^\times \cong \bigoplus_{p \in P} \Z
	\end{equation}
	by the prime factor decomposition. Thus, we need to find for each $p \in P$ an automorphism $\alpha_p \in \Aut(M_P \otimes \K)$ such that all of them commute and $[\alpha_p(1 \otimes e)] \in K_0(M_P)$ corresponds to $p$ under the isomorphism $K_0(M_P) \cong \Z_P$ induced by the trace. Fix $p$. We define $\bar{\alpha}_p$ as follows
	\[
		\begin{tikzcd}[column sep=1.5cm]
			M_p \otimes \K \ar[r,"\varphi_1 \otimes \id{\K}"] & M_p \otimes M_p(\C) \otimes \K \ar[r,"\id{M_p} \otimes \varphi_2"] & M_p \otimes \K
		\end{tikzcd}
	\]
	where $\varphi_1 \colon M_p \to M_p \otimes M_p(\C)$ and $\varphi_2 \colon M_p(\C) \otimes \K \to \K$ are isomorphisms. If $\tau$ is the trace on the finite rank projections of $M_p \otimes \K$ with $\tau(1 \otimes e) = 1$, then  $\tau(\bar{\alpha}_p(1 \otimes e)) = p$. We may view $M_P \otimes \K$ as the tensor product over all $M_p \otimes \K$ for all $p \in P$. In case $P$ is an infinite set we choose
	\[
		(M_{p_1} \otimes \K) \otimes \dots \otimes (M_{p_i} \otimes \K) \to (M_{p_1} \otimes \K) \otimes \dots \otimes (M_{p_i} \otimes \K) \otimes (M_{p_{i+1}} \otimes \K)
	\]
	to be given by $a \mapsto a \otimes (1 \otimes e)$. The automorphism $\alpha_p \in \Aut(M_P \otimes \K)$ is defined to act via $\bar{\alpha}_p$ on the appropriate tensor factor and the identity on the rest. These clearly commute.

	Set $D=M_P$ and
	recall that $\bar{G}_{D}(\mathbf{n})=\Aut_0((D\otimes \mathcal{K})^{\otimes n}).$
	
	There are maps of $\mI$-commutative monoids {\cite[Lem.6.2]{DP4}}:
\begin{equation}\label{eqn:mono} \bar{G}_D(\mathbf{n}) \to G_D(\mathbf{n}) \to K_0(D)^{\times}_{+}
\end{equation}
and a group homomorphism
\[K_0(D)^{\times}_{+} \times \Aut_0(D\otimes \mathcal{K})  \to \Aut((D\otimes \mathcal{K})\otimes (D\otimes \mathcal{K}))\cong \Aut(D\otimes \mathcal{K}),\]
given by $(x,\alpha)\mapsto \gamma(x)\otimes \alpha$.
From our previous discussion, this is a homotopy equivalence. We obtain a homotopy equivalence
\[B(K_0(D)^{\times}_{+}) \times B\Aut_0(D\otimes \mathcal{K})  \to B\Aut(D\otimes \mathcal{K}).\]
In conjunction with \eqref{eqn:mono}, this gives the exact sequence from the statement {(and a split as a sequence of pointed sets)}.
To see that the exact sequence of groups splits, one observes that if $X$ is a finite CW-complex, then
$H^1(X,(Z_P)_{+}^{\times})$ is a free group since it is isomorphic to $\mathrm{Hom}(H_1(X,\Z),(Z_P)_{+}^{\times})\cong ((Z_P)_{+}^{\times})^r$ where $r$ is the rank of $H_1(X,\Z)$.
\end{proof}

\begin{proposition}\label{prop:coeff} Let $P$ be a nonempty set of primes and let $X$ be a finite CW-complex.
The natural maps $\bE_{M_P}^*(X)\to\bE_{M_P}^*(X)\otimes \Z_P$ and $\bE_\mathcal{Z}^*(X)\otimes \Z_P \to  \bE_{M_P}^*(X)\otimes \Z_P$ are isomorphisms of groups. It follows that $\bE_{M_P}^*(X)\cong \bE_\mathcal{Z}^*(X)\otimes \Z_P$.
\end{proposition}
\begin{proof}

Let $D$ be a strongly self-absorbing $C^*$-algebra satisfying the UCT.
 Since $\Z_P$ is flat, it follows that $X \mapsto \bar{E}^*_D(X) \otimes Z_P$ still satisfies all axioms of a generalized cohomology theory on finite CW-complexes. We have natural transformations of cohomology theories:
 \[ \bE_{M_P}^*(X)\to\bE_{M_P}^*(X)\otimes \Z_P,\quad  \bE_\mathcal{Z}^*(X)\otimes \Z_P \to  \bE_{M_P}^*(X)\otimes \Z_P. \]
 These transformations induce isomorphisms of groups on coefficients by \eqref{homotopy1} and therefore they induce isomorphisms of cohomology theories. \end{proof}

Let us recall from \cite{DP1} that the natural map $D \to D \otimes \OO_\infty$ induces an isomorphism
\begin{equation}\label{reduced}
  \bar{E}^*_{D}(X) \stackrel{\cong}\longrightarrow \bar{E}^*_{D\otimes \OO_\infty}(X)
\end{equation}
since one checks that it induces an isomorphism on coefficients in view of \eqref{homotopy}.

{\begin{proposition}\label{prop:split}
  Let $X$ be a finite CW complex. The canonical maps $\hat{\rho} \colon\hE^1_{\C}(X)\to \hE^1_{\ZZ}(X)$
  and ${\rho} \colon E^1_{\C}(X)\to  E^1_{\ZZ}(X)$ induced by the unital $*$-homomorphism $\C\to \ZZ$  split naturally.  Consequently we have natural splittings \begin{equation}\label{h-splitting}
\hE^1_{\ZZ}(X)\cong \hE^1_{\C}(X) \oplus h(X)\quad \text{and} \quad E^1_{\ZZ}(X)\cong E^1_{\C}(X) \oplus h(X).\end{equation}
\end{proposition}
}
\begin{proof} By the CW-approximation theorem,
 there are the natural diagrams
\begin{equation}
\begin{tikzcd}
	\hE^1_{\C}(X) \ar[d,"\hat{r}_{\C}"] \ar[r,"\hat{\rho}"] & \hE^1_{\ZZ}(X) \ar[d,"\hat{r}_{\ZZ}"] \\
	\hE^1_{\C}(X_4) \ar[r, "\hat{\rho}_4"]  & \hE^1_{\ZZ}(X_4)
\end{tikzcd} \quad \begin{tikzcd}
	E^1_{\C}(X) \ar[d,"{r}_{\C}"] \ar[r,"{\rho}"] & E^1_{\ZZ}(X) \ar[d,"{r}_{\ZZ}"] \\
	E^1_{\C}(X_4) \ar[r, "{\rho}_4"]  & E^1_{\ZZ}(X_4)
\end{tikzcd}
\end{equation}
  induced by the inclusion of skeleta $ X_4\hookrightarrow X$.
  We will verify that the maps $\hat{\rho}_4, \hat{r}_\C,$ $\rho_4$ and ${r_\C}$  are bijections.
  This will clearly imply \eqref{h-splitting}  with
  {$$h(X):=\ker \hat{r}_{\ZZ}\cong \ker  r_{\ZZ}.$$}
   The isomorphism
  $\ker \hat{r}_{\ZZ}\cong \ker  r_{\ZZ}$ follows from the naturality of \eqref{short}.

  Consider the commutative diagram from the proof of \cite[Thm.~6.7]{DP4}:
  \[
		\begin{tikzcd}
		0 \ar[r]&	E^1_{\ZZ}(X) \ar[r] & \hE^1_{\ZZ}(X) \ar[r] &  H^1(X;\Z/2) \ar[r] & 0\\
		0 \ar[r]& E^1_\C(X) \ar[r]\ar[u,"\rho"]& \hE^1_\C(X)  \ar[r] \ar[u,"\hat\rho"]&   H^1(X;\Z/2)\ar[equal]{u}\ar[r] & 0\\
		0 \ar[r]& H^3(X;\Z) \ar[r]\ar[equal]{u} & H^3(X;\Z)\times_{_{tw}}   H^1(X;\Z/2) \ar[r] \ar[equal]{u}&   H^1(X;\Z/2)\ar[equal]{u}\ar[r] & 0
		\end{tikzcd}
	\]
 If we show that $\rho_4:E^1_\C(X_4) \to E^1_\ZZ(X_4)$ is bijective so is $\hat\rho_4.$
The map $\rho$ is induced by the map $B\Aut(\K)\to B\Aut(\ZZ\otimes \K)$ which is $4$-connected by the computations of \cite{DP3}, see \eqref{homotopy}. By Whitehead's theorem this shows that $\rho_4$ is surjective and that $\rho_3:E^1_\C(X_3) \to E^1_\ZZ(X_3)$ is bijective.
The restriction map $r'_\C \colon E^1_{\C}(X_4) \to E^1_{\C}(X_3)$ in the commutative diagram below
\begin{equation}
\begin{tikzcd}
	E^1_{\C}(X_4) \ar[d,"r'_\C"] \ar[r,"\rho_4"] & E^1_{\ZZ}(X_4)\ar[d,"r'_\ZZ"] \\
	E^1_{\C}(X_3) \ar[r,"\rho_3" above, "\cong" below]  & E^1_{\ZZ}(X_3)
\end{tikzcd}
\end{equation}
is injective since it identifies with the map $H^3(X_4,\Z) \to H^3(X_3,\Z) $ which is injective since
$H^3(X_4/X_3,\Z)=0$. It follows that $\rho_4$ is also injective. Next we show that the map $\hat r_\C$ is bijective.
Using the naturality of the exact sequence from Theorem~\ref{thm:Z2}, since the restriction map  $H^1(X;\Z/2)\to H^1(X_4;\Z/2)$ is bijective, it suffices to show that the restriction map $ r_\C:E^1_\C(X)\to E^1_\C(X_4)$ is bijective.
This follows from the exact sequence
\[H^3(X/X_4,\Z) \to H^3(X,\Z) \to H^3(X_4,\Z) \to H^4(X/X_4,\Z),\]
since $X/X_4$ is 4-connected.
\end{proof}

\begin{proposition}\label{prop:basic}
Let $P\neq \emptyset$ be a set of prime numbers and let $X$ be a finite CW complex. Then there are isomorphisms of groups
\begin{itemize}
\item[(a)]
\(E^1_\ZZ(X)\cong\bE^1_\ZZ(X)\cong H^3(X,\Z)\oplus h(X).\)
\item[(b)]
\(\hE^1_\ZZ(X)\cong (H^1(X,\Z/2)\times_{_{tw}} H^3(X,\Z))\oplus h(X).\)
\item[(c)]
\(E^1_{M_P}(X)\cong H^1(X, (Z_P)^\times_+) \oplus \bE^1_{M_P}(X)\cong H^1(X, (Z_P)^\times_+) \oplus H^3(X,\Z) \otimes \Z_P  \oplus h(X) \otimes \Z_P \)
\item[(d)]
\( \hE^1_{M_P}(X)\cong H^1(X, (Z_P)^\times_+)\oplus \left( H^1(X;\Z/2) \times_{_{tw}} H^3(X,\Z_P)\right)\oplus h(X) \otimes \Z_P\)
with multiplication on $ H^1(X;\Z/2) \times H^3(X,\Z_P)$
\[
	(w, \tau) \cdot (w',\tau') = (w + w', \tau + \tau' + \beta_P(w \cup w'))
\]
for $w,w' \in H^1(X,\Z/2)$ and $\tau,\tau' \in H^3(X,\Z_P)$, where $\beta_P \colon H^2(X,\Z/2) \to H^3(X,\Z_P)$ is the composition of the Bockstein homomorphism with the coefficient map $H^3(X,\Z)\to H^3(X,\Z_P)$.
\end{itemize}
\end{proposition}
\begin{proof}
The isomorphism (a) follows from \eqref{h-splitting}. {Recall that $E^1_\ZZ(X)\cong\bE^1_\ZZ(X)$ since $\Aut(\ZZ \otimes\K)$ is path connected.} The isomorphism (c) follows from (a), Proposition~\ref{thm:MP} and Proposition~\ref{prop:coeff}.

It remains to deal with (b) and (d). Along the way we shall review the proof of Theorem~\ref{thm:basic}.
If $D'\mapsto D$ is a unital $*$-monomorphism of strongly self-absorbing $C^*$-algebras, by \cite[Lem.~6.3]{DP4}
there is a commutative diagram of commutative  $\mI$-monoids:
\[
		\begin{tikzcd}
			G_D(\mathbf{n}) \ar[r] & \Ggr_D(\mathbf{n}) \ar[r] &  \Z/2\\
		G_{D'}(\mathbf{n}) \ar[r]\ar[u] & \Ggr_{D'}(\mathbf{n}) \ar[r] \ar[u]&  \Z/2\ar[equal]{u}
		\end{tikzcd}
	\]
which induces a commutative diagram
	\[
		\begin{tikzcd}
		0 \ar[r]&	E^1_{M_P}(X) \ar[r] & \hE^1_{M_P}(X) \ar[r] &  H^1(X;\Z/2) \ar[r] & 0\\
		0 \ar[r]&	E^1_\ZZ(X) \ar[r] \ar[u]& \hE^1_\ZZ(X) \ar[r] \ar[u]&  H^1(X;\Z/2) \ar[r] \ar[equal]{u}\ar[r] & 0\\
		0 \ar[r]& E^1_\C(X) \ar[r]\ar[u] & \hE^1_\C(X)  \ar[r] \ar[u]&   H^1(X;\Z/2)\ar[equal]{u}\ar[r] & 0\\
		0 \ar[r]& H^3(X;\Z) \ar[r]\ar[equal]{u} & H^3(X;\Z)\times  H^1(X;\Z/2) \ar[r] \ar[equal]{u}&   H^1(X;\Z/2)\ar[equal]{u}\ar[r] & 0
		\end{tikzcd}
	\]

 From Proposition~\ref{prop:Z2}, Remark~\ref{remark:review} and the diagram above we obtain :
\[\hE^1_\ZZ(X) \cong H^1(X;\Z/2) \times E^1_\ZZ(X) \cong \big(H^1(X;\Z/2) \times_{_{tw}} H^3(X,\Z)\big)\oplus h(X),\]
with multiplication on the first two factors as in Proposition~\ref{prop:Z2}. This proves part (b).

Let  $j: H^3(X,\Z)= E^1_{\C}(X)\to E^1_{M_P}(X)$ be the map induced by unital $*$-homomorphism $\C \to M_P$.
From the diagram above, Proposition~\ref{prop:Z2} and Remark~\ref{remark:review} we obtain that
\begin{equation}\label{eqn:ZZZP}
  \hE^1_{M_P}(X)\cong H^1(X;\Z/2) \times_{_{tw}}  E^1_{M_P}(X)
\end{equation}
where the group structure is given by $(w, x) \cdot (w', x') = (w + w', x + x' + j(\beta(w \cup w')))$
for $w,w' \in H^1(X,\Z/2)$ and $x,x' \in E^1_{M_P}(X)$. Note that the image of $j$ is contained in $\bE^1_{M_P}(X)$
since $E^1_{\C}(X)=\bE^1_{\C}(X)$.
Using the isomorphism:
\[E^1_{M_P}(X)\cong H^1(X, (Z_P)^\times_+) \oplus \bE^1_{M_P}(X)\cong H^1(X, (Z_P)^\times_+) \oplus\bE^1_\ZZ(X)\otimes \Z_P\] and the previous discussion,
we can identify that map $j: E^1_\C(X)\to E^1_{M_P}(X)$ with the  map
\[H^3(X,\Z)  \to H^1(X, (Z_P)^\times_+) \oplus H^3(X,\Z) \otimes \Z_P\oplus h(X) \otimes \Z_P,\]
induced by $H^3(X,\Z) \to H^3(X,\Z) \otimes \Z_P$, $h\mapsto h\otimes 1$.
Part (d) follows now from \eqref{eqn:ZZZP}.
\end{proof}

\subsection{Comparing {$E^1_\ZZ(X)$ and $bsu^1_{\otimes}(X)$ and proof of Theorem A}}
In this section we will identify the summand $h(X)$ of $E^1_{\ZZ}(X)$ from Prop.~\ref{prop:basic} with the first group of the generalised cohomology theory $bsu_\otimes^*(X)$. To understand this group note that the spaces $BU$ and $BSU$ both have two $H$-space structures: one arising from the direct sum and another one from the tensor product. To distinguish them we will denote the second one by $BU_{\otimes}$ and $BSU_{\otimes}$, respectively. It was first observed by Segal in \cite{paper:SegalCatAndCoh} that $BU_{\otimes}$ is in fact an infinite loop space. In particular, there is a cohomology theory $X \mapsto bu_{\otimes}^*(X)$, such that $bu_{\otimes}^0(X) = [X, BU_{\otimes}]$. This was later understood by May in \cite{May:E-infinity} to fit into a much richer theory of units for $E_{\infty}$-ring spectra. $K$-theory provides such an $E_\infty$-ring spectrum $KU$, which has a unit spectrum $gl_1(KU)$, whose $0$-connected cover is $sl_1(KU) \simeq bu_{\otimes}$ and whose $2$-connected cover gives $bsu_{\otimes}$.

The results in \cite{May:E-infinity} are phrased in the language of $S$-modules and $S$-algebras and not in terms of symmetric spectra. Thus, in principle the cohomology theories represented by the spectrum $gl_1(KU^\C)$ constructed in Sec.~\ref{subsec: units-KUd} could differ from $gl_1(KU)$ for an $S$-algebra $KU$ representing topological $K$-theory. We will see in Sections~5 that this is not the case by proving a strong uniqueness result.
{In particular,  Corollary~\ref{cor:appendix} establishes a natural isomorphism
\begin{equation}\label{eqn:basiccc}
	gl_1(KU^\C)^1(X) \cong gl_1(KU)^1(X),
\end{equation}
 which implies the following proposition:}

\begin{proposition}\label{prop:split_bsu}
There is a natural isomorphism $\hE^1_{\ZZ}(X)\cong \hE^1_{\C}(X) \oplus bsu^1_{\otimes} (X)$.
\end{proposition}
\begin{proof}  Consider the commutative diagram
\begin{equation}
\begin{tikzcd}
	\hE^1_{\ZZ}(X) \ar[d,"\hat{r}_\ZZ"] \ar[r,"\cong"] & gl_1(KU^\ZZ)^1(X)\ar[d] & gl_1(KU^\C)^1(X) \ar[d] \ar[l,"\cong" above] \\
	\hE^1_{\ZZ}(X_4) \ar[r,"\cong"] & gl_1(KU^\ZZ)^1(X_4) & gl_1(KU^\C)^1(X_4) \ar[l,"\cong" above]
\end{tikzcd}
\end{equation}
The horizontal arrows are isomorphisms by \eqref{eq:useful} and \eqref{eqn:key-unit}. It follows that the vertical arrows have isomorphic kernels. {Thus $h(X)=\ker\hat{r}_\ZZ$ is isomorphic to   the kernel of the map $gl_1(KU^\C)^1(X)\to gl_1(KU^\C)^1(X_4)$ and which in its turn, by Corollary~\ref{cor:appendix}, is isomorphic to the kernel of the map $gl_1(KU)^1(X)\to gl_1(KU)^1(X_4)$ which is the map $bu^1_\otimes (X)\to bu^1_\otimes (X_4)$ and hence isomorphic to $bsu^1_{\otimes} (X)$.}
\end{proof}
{As a consequence of the Adams-Priddy result on the uniqueness of $bsu$ which we review in Sec.~{4.1}, for a finite CW-complex $X$ there is a (nonnatural) isomorphism
\begin{equation}\label{eqn:basicc}
	bsu_{\otimes}^*(X)\cong  bsu_{\oplus}^*(X).
\end{equation}}
 {Using this in conjunction with the results from the previous sections, we derive our main result (Theorem A from introduction) which we restate here for the convenience of the reader.}
{\begin{theorem} \label{thm:main-intro1} Let $X$ be a finite  CW-complex and let $P$ be a nonempty set of prime numbers. There are (not natural) isomorphisms
\begin{itemize}
	\item[(a)] ${E}^1_{\ZZ}(X)\cong  H^3(X,\Z)\oplus k^5(X)$.
	\item[(b)] ${E}^1_{M_P}(X)\cong  H^1(X,(\Z_P)_+^\times){\oplus} H^3(X,\Z_P)\oplus k^5(X,\Z_P)$
	\item[(c)] ${E}^1_{\OO_\infty}(X)\cong \big(H^1(X, \Z/2)\times_{_{tw}}  H^3(X,\Z)\big)\oplus k^5(X)$.
    \item[(d)] ${E}^1_{M_P \otimes \OO_\infty}(X)\cong H^1(X,(\Z_P)_+^\times)\oplus \big(H^1(X, \Z/2)\times_{_{tw}} H^3(X,\Z_P)\big)\oplus k^5(X,\Z_P)$
\end{itemize}
The (twisted) multiplication on $ H^1(X;\Z/2) \times H^3(X,\Z_P)$ is given by
\begin{equation}\label{xcxc1}
	(w, \tau) \cdot (w',\tau') = (w + w', \tau + \tau' + \beta_P(w \cup w'))
\end{equation}
for $w,w' \in H^1(X,\Z/2)$ and $\tau,\tau' \in H^3(X,\Z_P)$, where $\beta_P \colon H^2(X,\Z/2) \to H^3(X,\Z_P)$ is the composition of the Bockstein homomorphism $\beta$ with the coefficient map $H^3(X,\Z)\to H^3(X,\Z_P)$. The multiplicaton in (c) is just like in ~\eqref{xcxc1} with   $\beta_P$
 replaced by $\beta \colon H^2(X,\Z/2) \to H^3(X,\Z)$.
\end{theorem}}

\begin{proof}
  By \eqref{eqn:basicc} we have  $bsu_{\otimes}^*(X)\otimes \Z_P\cong  bsu_{\oplus}^*(X)\otimes \Z_P$ for any set of primes $P$.
   It follows by \eqref{bbssuu} that $bsu_{\otimes}^1(X)\otimes \Z_P\cong k^5(X)\otimes \Z_P \cong k^5(X, \Z_P)$.
  Thus the statement follows now from Propositions~\ref{prop:split_bsu} and ~\ref{thm:basic} as $h(X)\cong bsu_{\otimes}^1(X)\cong k^5(X)$.
\end{proof}
The map $\beta_P$ vanishes if $H^3(X,\Z)$ has no $2$-torsion or if $2\in P$.

{The  proof of Theorem~\ref{thm:main-intro1} relies on the isomorphisms \eqref{eqn:basicc} and \eqref{eqn:basiccc} that will discussed in the next two sections.}

\section{$K$-theory as a commutative $S$-algebra and its units}\label{Sect.4}

The tensor product gives topological $K$-theory the structure of a multiplicative cohomology theory. The graded commutative multiplication lifts to the level of spectra in the sense of stable homotopy theory. In fact, $K$-theory can be represented by an $E_\infty$-ring spectrum. There are several approaches to make this precise. We have already met the commutative symmetric ring spectrum $KU^\C$. As we have seen in the previous sections, $KU^\C$ is closely linked to the classification of $C^*$-algebra bundles. In this section we give a brief overview of $K$-theory as a commutative $S$-algebra. A lot of the results about the infinite loop spaces $BU_{\oplus}$ and $BU_{\otimes}$ are easiest to prove using this approach. We will only give a brief overview of $S$-modules and $S$-algebras here and refer the reader to \cite{EKMM} for a complete reference.

Constructing a symmetric monoidal category of spectra with the sphere spectrum $S$ as its unit object is quite an intricate endeavour. The category $\OSpec[\bL]$ of (coordinate free) spectra that are also algebras over the linear isometries operad $\bL$ has almost all of the desired properties, in particular a symmetric monoidal structure given by a smash product $\wedge$ (see \cite[Section I.5]{EKMM}). The only defect is that the sphere spectrum is not a unit object. This can be fixed by restricting to the full subcategory $\SMod \subset \OSpec[\bL]$, on which $S$ acts like a unit, and the objects of $\SMod$ are called $S$-modules \cite[Def.~II.1.1]{EKMM}.

A \emph{(commutative) $S$-algebra} is a (commutative) monoid with respect to $\wedge$ in $\SMod$. We will write $\ComS$ for the category of commutative $S$-algebras \cite[Section~II.3]{EKMM}. Both, $\SMod$ and $\ComS$, are model categories \cite[Section~VII.4]{EKMM}; the weak equivalences are the morphisms that induce an isomorphism on $\pi_*$.

A commutative $S$-algebra representing connective $K$-theory can be constructed from the bipermutative category of finite-dimensional complex inner product spaces and unitary isomorphisms as follows: Let $\catU$ be the topological category with objects $\N_0$, i.e.\ the natural numbers including~$0$, where we think of $n \in \N_0$ as $\C^n$. The morphism spaces are given by
\[
	\hom(m,n) = \begin{cases}
		\emptyset & \text{if } m \neq n\ , \\
		U(n) & \text{if } m = n\ ,
	\end{cases}
\]
where $U(0)$ is the trivial group. The sum and product operations on $\N_0$ extend to the morphisms of $\catU$ via the block sum $\oplus$ and the Kronecker product $\otimes$ of unitary matrices. This gives $(\catU, \oplus, \otimes)$ the structure of a bipermutative category.

Combining \cite[Section~3, discussion on p.~24]{May:mult_inf_loop} and \cite[Cor.~3.6]{EKMM} shows that there is a commutative $S$-algebra $ku$ associated to $\catU$ that represents connective complex topological $K$-theory as a multiplicative cohomology theory by \cite[VIII.2.1]{May:E-infinity}. Its periodic counterpart, the commutative $S$-algebra $KU$, can be constructed by inverting the Bott element $\beta \in \pi_2(ku)$. (In fact, $KU$ is even a commutative $ku$-algebra by \cite[Thm.~VIII.4.3]{EKMM}.)

The $0$th space of an $\Omega$-spectrum is an infinite loop space. Extending this observation to the category $\OSpec[\bL]$ we obtain a functor $\Omega^\infty \colon \OSpec[\bL] \to \mathcal{T}[\bL]$ with codomain given by the $\bL$-algebras in the category of based topological spaces $\mathcal{T}$. Let $R \in \ComS \subset \OSpec[\bL]$ be a commutative $S$-algebra. The abelian group $\pi_0(\Omega^\infty R)$ turns out to be a ring with respect to the multiplication inherited from $R$. The space $GL_1(R)$ of units of $R$ is defined by the pullback diagram
\begin{equation} \label{eqn:space_of_units}
\begin{tikzcd}
	GL_1(R) \ar[r] \ar[d] & \Omega^\infty R \ar[d] \\
		GL_1(\pi_0(\Omega^\infty R)) \ar[r] & \pi_0(\Omega^\infty R)
\end{tikzcd}
\end{equation}
Thus, $GL_1(R)$ consists of those components in $\Omega^\infty R$ that are invertible with respect to the multiplication of $R$. The space $GL_1(R)$ turns out to be an $E_{\infty}$-ring space. As such it gives rise to a connective spectrum $gl_1(R)$ such that $\Omega^\infty gl_1(R) = GL_1(R)$. This construction can be made functorial and takes values in the category $\wSpec$ of \emph{weak $\Omega$-spectra}.

In contrast to highly structured spectra a weak $\Omega$-spectrum just consists of a sequence of based spaces $(X_n)_{n \in \N_0}$ together with structure maps $X_n \to \Omega X_{n+1}$ that are weak homotopy equivalences. A map of weak $\Omega$-spectra is a sequence of maps $(f_n)_{n \in \N_0}$ with $f_n \colon X_n \to Y_n$ compatible with the structure maps. An equivalence of weak $\Omega$-spectra $X,Y$ is a chain of such maps
\[
	\begin{tikzcd}
		X  & \ar[l] Z^{(1)} \ar[r] & Z^{(2)} & \ar[l] Z^{(3)} \ar[r] & \dots  & \ar[l] Z^{(k)} \ar[r] & Y
	\end{tikzcd}
\]
that are levelwise weak equivalences \cite[p.~209]{MayThomason:uniqueness}. For details about the functor
\[
	gl_1 \colon \text{Com}_{S} \to \wSpec
\]
from commutative $S$-algebras to the category $\wSpec$ we refer the reader to Section~4 and 5 in \cite{AndoBlumbergGepnerHopkinsRezk:Units} which is based on the treatment of units in \cite{May:E-infinity}.

If we replace the full group of invertible elements in $\pi_0(\Omega^\infty R)$ in the pullback diagram defining $GL_1(R)$ by the trivial subgroup (consisting only of the unit element of $\pi_0(\Omega^\infty R)$), then we obtain the space of special units $SL_1(R)$ and a corresponding connective spectrum $sl_1(R)$ (see for example \cite[Def.~7.6]{May:Einfty_spaces}). As a spectrum $sl_1(R) \to gl_1(R)$ is the $0$-connected cover.

\subsection{Localizations, connective $K$-theory, $bsu_{\oplus}$ and $bsu_{\otimes}$}\label{4.1}
Recall that we denote by $KU$ and $ku$ the commutative $S$-algebras representing periodic complex $K$-theory and its connective counterpart, respectively. The associated unit spectra are $gl_1(KU)$ and $gl_1(ku)$. The localisation map $ku \to KU$ that inverts the Bott element is a morphism of commutative $S$-algebras which induces an equivalence of the underlying infinite loop spaces and therefore also an equivalence
\[
	gl_1(ku) \to gl_1(KU)\ .
\]
We give some background and explain below how $gl_1(KU)$ is related to the spectrum $bu_{\otimes}$ appearing in \cite{May:E-infinity} (and, using slightly different machinery, in \cite{paper:SegalCatAndCoh}). We denote by $U$ the infinite unitary group, i.e.\ the colimit over the groups $U(n)$, and by $SU$ the infinite special unitary group, i.e.\ the colimit over the groups~$SU(n)$. As pointed out in the last section the operations of direct sum and tensor product on these groups are known to induce $H$-space structures on the corresponding classifying spaces. We indicate these structure by using the notation $BU_{\oplus}$, $BSU_{\oplus}$ and respectively $BU_{\otimes}$, $BSU_{\otimes}$.

Bott periodicity shows that $BU_{\oplus}$ and $BSU_{\oplus}$ are infinite loop spaces. The corresponding spectra are denoted by $bu_{\oplus}$ and $bsu_{\oplus}$. Note that $GL_1(\pi_0(\Omega^\infty KU)) \cong  GL_1(\pi_0(BU \times \Z)) \cong \Z/2\Z$. Hence, we obtain
\begin{align*}
	SL_1(KU) &\simeq BU_{\otimes} \ ,\\
	GL_1(KU) &\simeq \Z/2\Z \times BU_{\otimes} \ .
\end{align*}
Note that the first equivalence can be used to equip $BU_{\otimes}$ with an infinite loop space structure. We will therefore denote the spectrum $sl_1(KU)$ also by $bu_{\otimes}$. Let $bsu_{\otimes}$ be the $2$-connected cover of the spectrum $bu_{\otimes}$ and note that $\Omega^\infty bsu_{\otimes} \simeq BSU_{\otimes}$. Denote by $BGL_1(KU)$ the first delooping of $GL_1(KU)$, i.e.\ the first space in the sequence forming the spectrum $gl_1(KU)$. Likewise, let $BBSU_{\otimes}$ be the first delooping of $BSU_{\otimes}$ with respect to the spectrum $bsu_{\otimes}$. By \cite[Lem.~V.3.1]{May:E-infinity} (see also \cite[p.406]{MST}) we have a splitting of infinite loop spaces
\begin{equation} \label{K}
	BU_{\otimes} \simeq K(\Z,2) \times BSU_{\otimes}\ .
\end{equation}
 While it is true that the spaces $GL_1(KU)$ and  $BGL_1(KU)$ also decompose as products
\begin{align*}
	GL_1(KU) &\simeq \Z/2 \times BU_{\otimes}  \simeq \Z/2 \times K(\Z,2) \times BSU_{\otimes}\ , \\
	BGL_1(KU) &\simeq K(\Z/2,1) \times K(\Z,3) \times BBSU_{\otimes}\ ,
\end{align*}
these decompositions do not respect the infinite loop space structure {as  noted implicitly in \cite{paper:AtiyahSegal}.}
{Indeed, as we verify in the paper
\[[X,BGL_1(KU) ] \cong (H^1(X,\Z/2)\times_{_{tw}} H^3(X,\Z))\oplus bsu_{\otimes}^1(X)\]
which explains the twisting of the multiplication in Theorem A.}

 By a classic result of Adams and Priddy \cite{AP} (see also \cite[Thm.~V.4.2]{May:E-infinity}, or \cite[Cor.~10.3]{May:Einfty_spaces} for a more general statement), $BSU_{\oplus}$ and $BSU_{\otimes}$ become equivalent as infinite loop spaces on localization at any prime $p$.
\begin{theorem}[\cite{AP}]\label{thm:AP} There is an equivalence of infinite loop spaces
  \[(BSU_{\oplus})_{(p)}\simeq (BSU_{\otimes})_{(p)}\ .\]
\end{theorem}
Thus, the corresponding spectra $bsu_{\oplus}$ and $bsu_{\otimes}$ become equivalent as spectra on localisation at any prime $p$. This also turns out to be true for the completions at any prime $p$, but we will not need that statement.

If we are only interested in computing the groups $bsu_{\otimes}^*(X)$ and neglect naturality, then we can use the following observation: Let $X$ be a space with the homotopy type of a finite CW-complex. By Theorem \ref{thm:AP} there is a natural isomorphism $bsu_{\oplus}^*(X,\Z_{(p)})\cong bsu_{\otimes}^*(X,\Z_{(p)})$. Since $\Z_{(p)}$ is flat, by the universal coefficient theorem for generalized cohomology theories \cite{Adams:lectures_UCT}, $bsu_{\oplus}^*(X,\Z_{(p)})\cong bsu_{\oplus}^*(X)\otimes \Z_{(p)}$ and $bsu_{\otimes}^*(X,\Z_{(p)})\cong bsu_{\otimes}^*(X)\otimes \Z_{(p)}$. Two finitely generated abelian groups which are isomorphic after localization at each prime are necessarily isomorphic as is apparent from the structure theorem of such groups. Therefore, for every finite CW-complex there is a (not natural) isomorphism
\begin{equation}\label{eqn:basic}
	bsu_{\otimes}^*(X)\cong  bsu_{\oplus}^*(X).
\end{equation}


\section{Uniqueness of $gl_1$ for $K$-theory spectra}\label{Sect.5}
We have seen two ways of defining the unit spectrum of topological $K$-theory: the first one starting from the $S$-algebra $KU$ and the second one from the commutative symmetric ring spectrum $KU^\C$. The output of both constructions is a weak $\Omega$-spectrum. Our goal in this section is to compare them. By \cite[VIII.2.1]{May:E-infinity} the spectrum $ku$ represents connective topological $K$-theory, so after inverting the Bott element to obtain $KU$, this spectrum represents periodic topological $K$-theory as defined by Atiyah and Hirzebruch. Throughout this section we will use both notations, $KU^*$ and $K^*$, interchangeably to denote the cohomology theory represented by~$KU$ and similarly, $K_*$ and $KU_*$ for the corresponding homology theory. We will prove:

\begin{theorem}\label{thm:uuu}
\label{n1}
Let $F$ be a commutative symmetric ring spectrum representing complex
topological $K$-theory as a multiplicative cohomology theory on finite
CW-complexes in the sense that there exists a natural multiplicative
isomorphism $F^*(X) \cong K^*(X)$.  Suppose also that $F$ is a positive
$\Omega$-spectrum. Then there is an equivalence of weak $\Omega$-spectra
between $gl_1(F)$ and $gl_1(KU)$.
\end{theorem}

\begin{remark}
By Lemma~\ref{lem:multiplicative} the commutative symmetric ring spectra $KU^\C$, $KU^\ZZ$ and $KU^{\OO_\infty}$ satisfy the hypotheses of the theorem.
\end{remark}

\begin{corollary}\label{cor:appendix}
	There is a natural isomorphism of cohomology theories on finite CW complexes
	\[
		gl_1(KU^\C)^* (X)\cong gl_1(KU)^* (X)\ .
	\]
\end{corollary}

\subsection{Change of categories.}

In the rest of this section we will work with commutative $S$-algebras
(\cite{EKMM}) rather than commutative symmetric ring spectra, because we need
to use facts from \cite[Chapters V and VIII]{EKMM} whose analogues for
symmetric spectra are not written down in the literature. In
subsections~\ref{o1}--\ref{s2} we will prove the following:


\begin{theorem}
\label{n2}
Let $G$ be a commutative $S$-algebra that represents complex
topological $K$-theory as a multiplicative cohomology theory on finite
CW-complexes. Then there is a chain of weak equivalences of
commutative $S$-algebras between $G$ and $KU$.
\end{theorem}

In this subsection we show that Theorem \ref{n2} implies Theorem \ref{n1}.

We will use a fact about the relevant model categories. First recall that if
$\mathcal C$ is a model category then the {\it homotopy category} Ho$\mathcal C$
is obtained by inverting the weak equivalences (\cite[Section
1.2]{Hovey}), so two objects in the model category become isomorphic in the
homotopy category if they are connected by a chain of weak equivalences.

Recall that we write
$\mathcal{S}p^{\Sigma}$ for the category of symmetric spectra; we will write
$\mathrm{Com}^{\Sigma}$ for the category of commutative
symmetric ring spectra.  These categories have model category structures given
in \cite[Sections 9 and 15]{MMSS}; the weak equivalences for these model
category stuctures are the stable equivalences.

The category of $S$-modules is denoted $\mathcal{M}_S$.  We will write
$\mathrm{Com}_S$ for the category of commutative $S$-algebras (\cite[Section
II.3]{EKMM}). These are model categories (\cite[Section
VII.4]{EKMM}); the weak equivalences are the
morphisms that induce an isomorphism of $\pi_*$.

The fact we need is

\begin{proposition}
\label{n3}
There is an equivalence of categories
\[
\Upsilon:\mathrm{Ho}\mathrm{Com}^{\Sigma}
\to
\mathrm{Ho}\mathrm{Com}_S
\]
If $F\in \mathrm{Com}^{\Sigma}$ then $F$ and $\Upsilon(F)$ represent the same
cohomology theory on finite CW-complexes.
\end{proposition}


\begin{proof}[Proof of Proposition \ref{n3}]
Let $\mathrm{Com}^\mathcal{I}$ be the category of commutative orthogonal
ring spectra (\cite[Example 4.4]{MMSS}).  By \cite[Theorem 0.7]{MMSS} there
is an equivalence of categories
\[
\Upsilon_1:\mathrm{Ho}\mathrm{Com}^{\Sigma}
\to
\mathrm{Ho}\mathrm{Com}^{\mathcal{I}}
\]
with the property that $F$ and $\Upsilon_1(F)$ represent the same
cohomology theory on finite CW-complexes.  By \cite[Theorem 1.5]{MM} there
is an equivalence of categories
\[
\Upsilon_2:\mathrm{Ho}\mathrm{Com}^{\mathcal{I}}
\to
\mathrm{Ho}\mathrm{Com}_S,
\]
and by \cite[Theorem 7.13]{MM} $F$ and $\Upsilon_2(F)$ represent the same
cohomology theory on finite CW-complexes.
\end{proof}

Before we continue we explain how Theorem~\ref{n2} implies Theorem~\ref{n1}.
Let $F$ be as in Theorem~\ref{n1} and note that $\Upsilon(F)$ is a commutative
$S$-algebra representing $K$-theory as a multiplicative cohomology theory.
The functor $\Upsilon$ is the composition of the functors
$\mathbf{L}\mathbb{P}$ and $\mathbf{L}\mathbb{N}$ in \cite[Prop.~13.9]{paper:LindUnits}
and \cite[Prop.~14.1]{paper:LindUnits}, respectively. Thus, combining these
two propositions we obtain an equivalence of weak $\Omega$-spectra
\[
	gl_1(\Upsilon(F_c)) \simeq gl_1(F)
\]
where $F_c$ is a cofibrant replacement of $F$ as a commutative symmetric
ring spectrum (which still represents $K$-theory as a multiplicative
cohomology theory). By Theorem~\ref{n2} there is a chain of
weak equivalences of commutative $S$-algebras between $\Upsilon(F_c)$
and $KU$. Since $gl_1$ preserves weak equivalences, it gives a chain
of equivalences of weak $\Omega$-spectra
\[
	gl_1(F) \simeq gl_1(\Upsilon(F_c)) \simeq gl_1(KU)\ .
\]

\subsection{Obstruction theory.}

\label{o1}

In subsections \ref{s1}--\ref{s2} we will prove:

\begin{lemma}
\label{f2}
Let $G$ be a commutative $S$-algebra representing complex
topological $K$-theory as a multiplicative cohomology theory on finite
CW-complexes.
Then there is an isomorphism from $G$ to
$KU$ in the homotopy category $\mathrm{Ho}\mathcal{M}_S$  which makes
the diagrams
\vspace*{2mm}
\[
\begin{tikzcd}
G\wedge G \ar[r] \ar[d] & KU\wedge KU \ar[d]\\
G \ar[r] & KU
\end{tikzcd}
\]
and
\[
\begin{tikzcd}
& S^0\ar[ld] \ar[rd] & \\
G\ar[rr] & & KU
\end{tikzcd}
\]
commute in $\mathrm{Ho}\mathcal{M}_S$.
\end{lemma}

In this subsection we use Lemma \ref{f2} to prove Theorem \ref{n2}.

We use the obstruction theory of Goerss and Hopkins \cite{GH},
specifically
Corollary 5.9 of \cite{GH} with $E=KU$ and $A=KU_*KU$.%
\footnote{Section 1 of
\cite{GH} explains that this obstruction theory applies to the category of
commutative $S$-algebras.}

First we consider the obstruction groups in that Corollary:
\[
D^{n+1}_{E_*T/E_*E}(A,\Omega^n A).
\]
According to \cite[Theorem 2.6]{BR}, these groups are isomorphic to the
Gamma cohomology groups
\[
H\Gamma^{n+1}(A|E_*,\Omega^n A),
\]
and according to \cite[Theorem 6.2]{paper:BakerRichter} these groups
(with our choice of $A$ and $E$) are 0.

Next we observe that $KU_* G$ is isomorphic to $KU_*KU$ as a
commutative algebra in the category of $KU_*KU$ comodules; this is
immediate from Lemma \ref{f2} and the definition of the comodule structures
(\cite[page 281] {Adams_Stable}).

Now \cite[Corollary 5.9]{GH} and the definition of $\mathcal{TM}(A)$ on page
183 of \cite{GH} give a diagram of commutative $S$-algebras
\begin{equation}
\label{f3}
G\to G_1 \leftarrow G_2 \to \cdots \leftarrow KU
\end{equation}
in which each map is a $KU_*$ isomorphism.

Next we apply $KU$-localization (see \cite[Sections VIII.1 and VIII.2]{EKMM}).
Theorem VIII.2.2 of \cite{EKMM} says that if $H$ is a cell commutative
$S$-algebra (see \cite[Definition VII.4.11]{EKMM}) then there is a
$KU$-localization $\lambda\colon H\to \overline{H}$ with the property that $\overline{H}$ is a
commutative $S$-algebra and $\lambda$ is a map of commutative
$S$-algebras.  In order to apply this theorem we need to know that for every
commutative $S$-algebra $J$ there is a cell commutative $S$-algebra $CJ$ and a
weak equivalence of commutative $S$-algebras $\kappa \colon CJ\to J$, and moreover this
construction gives a functor $C \colon \mathrm{Com}_S\to \mathrm{Com}_S$ and a natural
transformation $\kappa$ from $C$ to the identity functor (see \cite[Lemma
5.8]{MMSS} and \cite[Lemma VII.5.8]{EKMM}).

Now consider the diagram
\begin{equation} \label{f4}
\begin{tikzcd}
	G \ar[r] & G_1 & G_2 \ar[r] \ar[l] & \cdots & \ar[l] KU \\
	CG \ar[r] \ar[d, "\lambda" left] \ar[u, "\kappa"] & CG_1 \ar[d,"\lambda" left] \ar[u,"\kappa"] &
	CG_2 \ar[d,"\lambda" left] \ar[u,"\kappa"] \ar[r] \ar[l] &  \cdots & \ar[l] \ar[d,"\lambda" left] \ar[u,"\kappa"] CKU \\
	\overline{CG} \ar[r] & \overline{CG_1} & \overline{CG_2} \ar[r] \ar[l] &  \cdots & \ar[l] \overline{CKU}
\end{tikzcd}
\end{equation}
The maps in the third row are given by Theorem VIII.2.2 of \cite{EKMM}, and the
lower half of the diagram homotopy commutes.
Because the $\kappa$ are weak equivalences, all the maps in the second row are
$KU_*$-isomorphisms, and
because the $\lambda$ are $KU_*$-isomorphisms, all
the maps in the third row are $KU_*$-isomorphisms.  Then the maps in the third
row are weak equivalences, because a
$KU_*$-isomorphism between $KU$-local spectra is a weak equivalence.
The map $\lambda \colon CKU\to\overline{CKU}$ is also a weak equivalence, because $KU$ is
$KU$-local, and
Lemma \ref{f2} implies that $G$ is weakly equivalent (as an $S$-module) to $KU$,
and therefore $CG$ is $KU$-local and
the map $\lambda\colon CG\to \overline{CG}$ is a weak equivalence.
Now the diagram gives the chain of weak equivalences promised by Theorem~\ref{n2}.

\subsection{The $S$-module $\Sigma^\infty \mathbb{C}P^\infty_+[b^{-1}]$}
\label{s1}

Let $\xi$ be the canonical complex line bundle over $\mathbb{C}P^\infty$.  The
classifying map for $\xi$ is a homotopy equivalence from $\mathbb{C}P^\infty$ to the
classifying space $BS^1$.
The
multiplication of $S^1$ induces an associative and commutative multiplication on
$BS^1$ and this gives a homotopy associative and homotopy commutative
multiplication on $\mathbb{C}P^\infty$.  This in turn gives a homotopy associative and
homotopy commutative multiplication on the $S$-module $\Sigma^\infty
\mathbb{C}P^\infty_+$ (where $+$ denotes a disjoint basepoint).

Next recall that if $X$ is a based CW complex there is a natural isomorphism in
the homotopy category $\mathrm{Ho}\mathcal{M}_S$
\[
\nu\colon \Sigma^\infty X_+\to
\Sigma^\infty S^0
\vee
\Sigma^\infty X
\]
for which the composite
\[
\Sigma^\infty X_+\xrightarrow{\nu} \Sigma^\infty X \vee \Sigma^\infty S^0
\to \Sigma^\infty S^0
\]
is induced by the based map
\[
p \colon X_+\to S^0
\]
which takes $X$ to the non-basepoint,
and the composite
\[
\Sigma^\infty X_+\xrightarrow{\nu}
\Sigma^\infty S^0
\vee
\Sigma^\infty X
\to \Sigma^\infty X
\]
is induced by the based map
\[
q \colon X_+\to X
\]
which is the identity on $X$.  The inverse to $\nu$ is the map
\[
\Sigma^\infty S^0
\vee
\Sigma^\infty X
\xrightarrow{\Sigma^\infty p'\vee Q}
\Sigma^\infty X_+
\]
where $p'$ takes the non-basepoint of $S^0$ to the basepoint of $X$ and
$Q\circ \Sigma^\infty q$ is the map $1-\Sigma^\infty(p'\circ p)$.

Let
\[
a\colon S^2\to \mathbb{C}P^\infty
\]
be the inclusion of the 2-skeleton and let
\[
b \colon \Sigma^\infty S^2\to \Sigma^\infty \mathbb{C}P^\infty_+
\]
be the composite
\[
\Sigma^\infty S^2\xrightarrow{\Sigma^\infty a}
\Sigma^\infty \mathbb{C}P^\infty
\xrightarrow{Q}
\Sigma^\infty \mathbb{C}P^\infty_+
\]

Let us define the $S$-module $\Sigma^\infty
\mathbb{C}P^\infty_+[b^{-1}]$ to be the telescope
\[
\text{Tel } \Sigma^{-2n}\Sigma^\infty \mathbb{C}P^\infty_+
\]
where the map
\[
\Sigma^{-2n} \Sigma^\infty
\mathbb{C}P^\infty_+
\to
\Sigma^{-2n-2} \Sigma^\infty
\mathbb{C}P^\infty_+
\]
is left multiplication by $\Sigma^{-2}b\colon S^0\to \Sigma^{-2} \Sigma^\infty
\mathbb{C}P^\infty_+$
(cf.\ \cite[page 94]{EKMM}).

\begin{lemma}
\label{p1}
$
\pi_* (\Sigma^\infty
\mathbb{C}P^\infty_+[b^{-1}])
$
is 0 in odd degrees and in even degree $2n$ is a
copy of $\mathbb Z$ generated by $b^n$.
\end{lemma}

\begin{proof}
$
\pi_* (\Sigma^\infty
\mathbb{C}P^\infty_+[b^{-1}])
$
is isomorphic to $ (\pi_* \Sigma^\infty \mathbb{C}P^\infty_+)[b^{-1}]$
(cf.\ \cite[bottom of page 94]{EKMM}), so the lemma follows from a theorem of
Snaith
\cite[Theorem 2.12 and equation 2.4]{paper:Snaith}.
\end{proof}

\subsection{Equivalences from $\Sigma^\infty\mathbb{C}P^\infty_+[b^{-1}]$ to $KU$ and $G$.}
First we construct a weak equivalence
\[
\eta\colon \Sigma^\infty\mathbb{C}P^\infty_+[b^{-1}]\to KU
\]

The canonical bundle $\xi$ over $\mathbb{C}P^\infty$ restricts to a bundle
$\xi_n$ over $\mathbb{C}P^n$, which gives a sequence of elements
$x_n\in K^0(\mathbb{C}P^n)$, and this sequence gives an element $x\in
K^0(\mathbb{C}P^\infty)$ by \cite[Proposition III.8.1 and Exercise (ii) on page
222]{Adams_Stable}. $x$ gives a map
\[
\bar{x} \colon\Sigma^\infty\mathbb{C}P^\infty_+\to KU
\]

\begin{lemma}
\label{o2}
The composite
\[
\Sigma^\infty S^2 \xrightarrow{b}
\Sigma^\infty\mathbb{C}P^\infty_+
\xrightarrow{\bar{x}}
KU
\]
is $\pm \beta$, where $\beta$ is the Bott element.
\end{lemma}

\begin{proof}
It suffices to show that the element of $\tilde{K}^0(S^2)$ represented by this
composite is a generator.  With the notation of Subsection \ref{s1}, it
suffices to show that $Q^*$ takes the class $x_2$ to a
generator of $\tilde{K}^0(S^2)$.  In the direct sum diagram
\[
\begin{tikzcd}[column sep=1.4cm]
	  \tilde{K}^0(S^2) \arrow[r,shift left=2pt, "q^*" above]
	& \tilde{K}^0(S^2_+) \arrow[l,shift left=2pt, "Q^*" below]
	  \arrow[r, shift right=2pt, "(p')^*" below]
	& \tilde{K}^0(S^0) \arrow[l, shift right=2pt, "p^*" above]
\end{tikzcd}
\]
we have
\[
q^*Q^*(x_2)=x_2-p^*(p')^*(x_2)=x_2-1
\]
Now $q^*$ maps isomorphically to the image of $\text{id}-p^*(p')^*$ and $x_2-1$
generates the image of $\text{id}-p^*(p')^*$, so $Q^*(x_2)$ is a generator as
required.
\end{proof}

Using this lemma, there is a map
\[
\text{Tel } \Sigma^{-2n}\Sigma^\infty \mathbb{C}P^\infty_+
\to
KU
\]
whose restriction to $\Sigma^{-2n}\Sigma^\infty \mathbb{C}P^\infty_+$ is
the composite
\[
\Sigma^{-2n}\Sigma^\infty \mathbb{C}P^\infty_+
\xrightarrow{\Sigma^{-2n}\bar{x}}
\Sigma^{-2n}KU
\xrightarrow{
\beta^{-n}\cdot}
KU
\]
for each $n$,
and moreover this map is unique up to
homotopy by \cite[Proposition III.8.1 and Exercise (ii) on page
222]{Adams_Stable}.  Define
\[
\eta \colon \Sigma^\infty\mathbb{C}P^\infty_+[b^{-1}]\to KU
\]
to be this map.

Next, using the
fact that $G$ represents $K$-theory on finite
CW-complexes, the argument that constructed $\eta$
gives a map
\[
\bar{y} \colon\Sigma^\infty\mathbb{C}P^\infty_+\to G
\]
and a
weak equivalence
\[
\theta \colon \Sigma^\infty\mathbb{C}P^\infty_+[b^{-1}]\to G
\]

\subsection{A product map for $\mathbb{C}P^\infty_+[b^{-1}]$}

In this subsection we construct a map
\[
\mu: \Sigma^\infty
\mathbb{C}P^\infty_+[b^{-1}]
\wedge
\Sigma^\infty
\mathbb{C}P^\infty_+[b^{-1}]
\to
\Sigma^\infty
\mathbb{C}P^\infty_+[b^{-1}]
\]
in $\mathrm{Ho}\mathcal{M}_S$.

First observe that there is a weak equivalence
\[
\text{Tel } ((\Sigma^{-2n} \Sigma^\infty
\mathbb{C}P^\infty_+) \wedge (\Sigma^{-2n} \Sigma^\infty
\mathbb{C}P^\infty_+))
\xrightarrow{\simeq}
(\text{Tel } \Sigma^{-2n} \Sigma^\infty
\mathbb{C}P^\infty_+)\wedge
(\text{Tel } \Sigma^{-2n} \Sigma^\infty
\mathbb{C}P^\infty_+)
\]
induced by the diagonal map of the unit interval.  Next observe that,
because the multiplication of
$\Sigma^\infty
\mathbb{C}P^\infty_+$ is homotopy associative and homotopy commutative,
there is a map
\[
m: \text{Tel } ((\Sigma^{-2n} \Sigma^\infty
\mathbb{C}P^\infty_+) \wedge (\Sigma^{-2n} \Sigma^\infty
\mathbb{C}P^\infty_+))
\to
\text{Tel } \Sigma^{-4n} \Sigma^\infty
\mathbb{C}P^\infty_+
\]
whose restriction to $(\Sigma^{-2n} \Sigma^\infty
\mathbb{C}P^\infty_+) \wedge (\Sigma^{-2n} \Sigma^\infty
\mathbb{C}P^\infty_+)$ is the composite
\[
(\Sigma^{-2n} \Sigma^\infty
\mathbb{C}P^\infty_+) \wedge (\Sigma^{-2n} \Sigma^\infty
\mathbb{C}P^\infty_+)
\to
\Sigma^{-4n} (\mathbb{C}P^\infty_+ \wedge \mathbb{C}P^\infty_+)
\to
\Sigma^{-4n} \Sigma^\infty
\mathbb{C}P^\infty_+
\]
for each $n$.  Moreover, up to homotopy there is only one map $m$ with this
property, by
\cite[Proposition III.8.1 and Exercise (ii) on page
222]{Adams_Stable} (using the fact that multiplication by $\beta$ is a weak
equivalence on the target of $m$).
Now let $\mu$ be the
composite
\begin{multline*}
\Sigma^\infty
\mathbb{C}P^\infty_+[b^{-1}]
\wedge
\Sigma^\infty
\mathbb{C}P^\infty_+[b^{-1}]
=
(\text{Tel } \Sigma^{-2n} \Sigma^\infty
\mathbb{C}P^\infty_+)\wedge
(\text{Tel } \Sigma^{-2n} \Sigma^\infty
\mathbb{C}P^\infty_+)
\simeq
\\
\text{Tel } ((\Sigma^{-2n} \Sigma^\infty
\mathbb{C}P^\infty_+) \wedge (\Sigma^{-2n} \Sigma^\infty
\mathbb{C}P^\infty_+))
\xrightarrow{m}
\text{Tel } \Sigma^{-4n} \Sigma^\infty
\mathbb{C}P^\infty_+
\simeq
\Sigma^\infty
\mathbb{C}P^\infty_+[b^{-1}]
\end{multline*}

\subsection{Proof of Lemma \ref{f2}}
\label{s2}

To prove Lemma \ref{f2} it suffices to show that the diagrams
\begin{equation} \label{p2}
\begin{tikzcd}
& \Sigma^\infty S^0 \ar[ld] \ar[rd] & \\
\Sigma^\infty\mathbb{C}P^\infty_+[b^{-1}] \ar[rr, "\simeq" below, "{\eta}" above] & & KU
\end{tikzcd}
\end{equation}
\begin{equation} \label{p3}
\begin{tikzcd}
\Sigma^\infty\mathbb{C}P^\infty_+[b^{-1}] \wedge
\Sigma^\infty\mathbb{C}P^\infty_+[b^{-1}]
\ar[rr, "\eta \wedge\eta" above, "\simeq" below]
\ar[d, "\mu"]
&&
KU\wedge KU
\ar[d]
\\
\Sigma^\infty\mathbb{C}P^\infty_+
\ar[rr, "\eta" above, "\simeq" below]
&&
KU
\end{tikzcd}
\end{equation}
\begin{equation} \label{p4}
\begin{tikzcd}
& \Sigma^\infty S^0 \ar[ld] \ar[rd]& \\
\Sigma^\infty\mathbb{C}P^\infty_+[b^{-1}] \ar[rr, "\simeq" below, "\theta" above] & & G
\end{tikzcd}
\end{equation}
and
\begin{equation} \label{p5}
\begin{tikzcd}
\Sigma^\infty\mathbb{C}P^\infty_+[b^{-1}] \wedge
\Sigma^\infty\mathbb{C}P^\infty_+[b^{-1}]
\ar[rr, "\theta \wedge\theta" above, "\simeq" below]
\ar[d, "\mu"]
&&
G\wedge G
\ar[d]
\\
\Sigma^\infty\mathbb{C}P^\infty_+
\ar[rr, "\theta" above, "\simeq" below]
&&
G	
\end{tikzcd}
\end{equation}
commute in $\mathrm{Ho}\mathcal{M}_S$.

We will give the proof for diagrams \eqref{p2} and \eqref{p3}; the proof for
the other two diagrams is the same.

\begin{lemma}
\label{o3}
The diagrams
\begin{equation} \label{o4}
\begin{tikzcd}
& \Sigma^\infty S^0 \ar[ld] \ar[rd]& \\
\Sigma^\infty\mathbb{C}P^\infty_+ \ar[rr, "\bar{x}" above] & & KU
\end{tikzcd}
\end{equation}
and
\begin{equation} \label{o5}
\begin{tikzcd}
\Sigma^\infty\mathbb{C}P^\infty_+ \wedge \Sigma^\infty\mathbb{C}P^\infty_+
\ar[rr, "\bar{x}\wedge \bar{x}" above]
\ar[d]
&&
KU\wedge KU
\ar[d]
\\
\Sigma^\infty\mathbb{C}P^\infty_+
\ar[rr, "\bar{x}" above]
&&
KU
\end{tikzcd}
\end{equation}
commute in $\mathrm{Ho}\mathcal{M}_S$.
\end{lemma}

\begin{proof}
Diagram \eqref{o4} commutes because the
restriction of $x$ to $\mathbb{C}P^0$ is the standard
generator of $K(\mathbb{C}P^0)$. For diagram \eqref{o5} we need to show that the pullback of $x$ along the map
\[
\Sigma^\infty\mathbb{C}P^\infty_+ \wedge \Sigma^\infty\mathbb{C}P^\infty_+
\to
\Sigma^\infty\mathbb{C}P^\infty_+
\]
is the exterior product $x\times x$. For this it suffices, by \cite[Proposition III.8.1 and Exercise (ii) on page
222]{Adams_Stable}, to show that the pullback of $x$ along the map
\[
\Sigma^\infty\mathbb{C}P^n_+ \wedge \Sigma^\infty\mathbb{C}P^n_+
\to
\Sigma^\infty\mathbb{C}P^\infty_+
\]
is $x_n\times x_n$ for each $n$.
The map
\[
\mathbb{C}P^n \times \mathbb{C}P^n \to
\mathbb{C}P^\infty \times \mathbb{C}P^\infty
\to \mathbb{C}P^\infty
\]
factors up to homotopy through ${\C}P^{2n}$ by cellular approximation, so
we have a homotopy commutative diagram
\begin{equation} \label{o6}
\begin{tikzcd}
\mathbb{C}P^n \times \mathbb{C}P^n \ar[d] \ar[r]&
\mathbb{C}P^\infty \times \mathbb{C}P^\infty
\ar[d]
\\
\mathbb{C}P^{2n}\ar[r] &  \mathbb{C}P^\infty
\end{tikzcd}
\end{equation}
In this diagram, the pullback of the bundle $\xi$ along the clockwise composite
is isomorphic to the pullback of the bundle $\xi$ along the counterclockwise
composite, so the pullback of $\xi_{2n}$ along the left vertical arrow is
$\xi_n\otimes \xi_n$.
Applying $K$-theory to diagram \eqref{o6} gives the commutative diagram
\[
\begin{tikzcd}
K^0(\mathbb{C}P^n \times \mathbb{C}P^n) &
K^0(\mathbb{C}P^\infty \times \mathbb{C}P^\infty) \ar[l]
\\
K^0(\mathbb{C}P^{2n})\ar[u] &  K^0(\mathbb{C}P^\infty)\ar[u]\ar[l]
\end{tikzcd}
\]
We have just shown that the image of $x_{2n}$ under that left vertical arrow is
$x_n\times x_n$, so the image of $x$ under the counterclockwise composite is
$x_n\times x_n$ as required.
\end{proof}

Now the commutativity of diagram \eqref{o4} implies that of \eqref{p2}, so it
only remains to show commutativity of diagram \eqref{p3}.  By
\cite[Proposition III.8.1 and Exercise (ii) on page
222]{Adams_Stable} it suffices to show that the two composites in diagram
\eqref{p3} have homotopic restrictions to
$(\Sigma^{-2n} \Sigma^\infty
\mathbb{C}P^\infty_+) \wedge (\Sigma^{-2n} \Sigma^\infty
\mathbb{C}P^\infty_+)$
for each $n$, and this follows from Lemma \ref{o3}.


\textbf{Acknowledgements}
 We would like to thank Mike Mandell for pointing out that the isomorphism \eqref{eqn:basic} follows from the Adams-Priddy theorem. The third author would like to thank John Lind for helpful discussions about the two incarnations of unit spectra used in this paper.


\end{document}